\documentclass[english]{article}
\usepackage{geometry}
\usepackage{amsmath,amssymb,latexsym}
\usepackage{enumerate}
\usepackage{graphicx,epsfig}
\usepackage{color}
\usepackage{eepic,epsfig}
\usepackage{fullpage}
\usepackage{here}

\newenvironment{proof}{{\bf{Proof}}\\}{$\square$\\}

\newtheorem{thm}{Theorem}[section]
\newtheorem{lem}{Lemma}[section]

\newtheorem{cor}{Corollary}[section]
\newtheorem{prop}{Proposition}[section]

\newtheorem{definition}{Definition}[section]
\newtheorem{rem}{Remark}[section]
\newtheorem{algo}{Algorithm}[section]
\newtheorem{exe}{Example}[section]

\def\Nset{\mathbb N}
\def\Rset{\mathbb R}
\def\Cset{\mathbb C}

\def\Pset{\mathbb P}

\def\lp{\left(}
\def\rp{\right)}

\numberwithin{equation}{section}

\def\lp{\left(}
\def\rp{\right)}

\def\bsm{ \lp \begin{smallmatrix} }
\def\esm{ \end{smallmatrix} \rp}

\def\PX{\mathbb P}
\def\PXz{\PX_0}

\def\PXi{\PX_{\infty}}

\def\PXg{\PX_{g}}
\def\PXr{\PX_{r}}

\begin{document}

\numberwithin{equation}{section}

\title{On a perturbation theory of Hamiltonian systems with periodic coefficients}

\date{ }
\author{Traor\'e  G. Y. Arouna, \thanks{%
 Universit\'e  F\'elix Houphou\"et-Boigny  de Cocody-Abidjan.
  Laboratoire de  Math\'ematiques Appliqu\'ees. UFR   Math\'ematiques   et Informatique,
  22  BP  582  Abidjan  22, C\^ote d'Ivoire (traorearounagogbeyves@yahoo.ca)}
\and  Mouhamadou Dosso\thanks{%
 Universit\'e F\'elix Houphou\"et-Boigny de cocody-Abidjan. Laboratoire de  Math\'ematiques Appliqu\'ees.
  UFR de Math\'ematiques et Informatique,  22 BP 582 Abidjan  22,
 C\^ote d'Ivoire (mouhamadou.dosso@univ-fhb.edu.ci or mouhamadoudoss@yahoo.fr)}
 ,\and  Jean-Claude  Koua Brou\thanks{%
  Universit\'e  F\'elix Houphou\"et-Boigny  de Cocody-Abidjan.
  Laboratoire de  Math\'ematiques Appliqu\'ees. UFR   Math\'ematiques   et Informatique,
  22  BP  582  Abidjan  22, C\^ote d'Ivoire ($k_{-}$brou@hotmail.com)}
  }

\maketitle

\begin{abstract}
A  theory of  rank $k\ge 2$  perturbation  of symplectic  matrices  and  Hamiltonian systems with periodic coefficients
  using    a base of  isotropic  subspaces, is presented.
After showing that the fundamental matrix ${\displaystyle \left(\widetilde{X}(t)\right)_{t\ge 0}}$  of  the rank $k$ perturbation
of  Hamiltonian system with periodic  coefficients
 and the rank $k$  perturbation of the fundamental  matrix  ${\displaystyle \left(X(t)\right)_{t\ge 0}}$  of the  unperturbed system are the same,
 the Jordan canonical form of ${\displaystyle \left(\widetilde{X}(t)\right)_{t\ge 0}}$   is given.
   Two  numerical examples   illustrating this theory and the consequences of  rank $k$ perturbations
    on the strong stability of Hamiltonian systems were  also given.
\vspace{0.5cm}


\begin{tabular}{lcl}
{\bf Key words} &:&  Hamiltonian system - symplectic matrix - isotropic subspace - perturbations - \\
 & &  strong stability  \\
 \end{tabular}
\end{abstract}

\section{Introduction}\label{Sec1}
The Hamiltonian systems  with periodic coefficients  are generally derived from
physical problems and engineering \cite{YS}.
 These  systems  are differential equations  with periodic coefficients
that originate from the theory of optimal control  \cite{Brez_02,hass_99}  and parametric resonance \cite{god-sad_06}.
They can be put in the form
\begin{equation}
J\frac{dx(t)}{dt}=H(t)x(t),\quad t\in \Rset
\label{Eq1}
\end{equation}
where $H(t)\in  \Rset^{2N\times 2N}$  is symmetric and $P$-periodic i.e.
 $H(t+P)=H(t)=(H(t))^T$  and $J$ is skew-symmetric matrix of $\Rset^{2N\times 2N}$.
 The square matrix
  $X(t)$  with  columns  $x_1(t),x_2(t),...,x_{2N}(t)$ belonging  to
  fundamental  set  of solutions  of equation  (\ref{Eq1}),
   is called  a fundamental  matrix.
   Considering the following matrix system \cite[Vol. 1, chap. 2]{YS}

  \begin{equation}\label{Eq2}
\left\{\begin{array}{rcl}
    J\dfrac{dX(t)}{dt}&=&H(t)X(t),\; t\in \Rset_{+}\\
    & & \\
     X(0)&=&I
     \end{array}\right.
\end{equation}
whose matrix solution  $\left(X(t) \right)_{t\in \Rset_+}$ satisfies the relationship
$X(t+nP)=X(t)X^n(P),\;\forall\; t\in\, \Rset_{+}$  and $\forall\;n\in \Nset.$
 We have the following definition
\begin{definition}
The  matrix $X(t)$  satisfying  equation (\ref{Eq2})  is called  the matrizant
    of equation  (\ref{Eq1}). The value at the period  $P$  of the matrizant $X(t)$
    defined  by  the initial condition  $X(0)=I_{2N}$, is called
  the monodromy  matrix  and its eigenvalues  are the multipliers  of system (\ref{Eq1}).
  \end{definition}
  An important property of Hamiltonian system with periodic coefficients is that the
     matrizant $X(t)\in \Rset^{2N\times 2N}$  of $(\ref{Eq2})$ verifies the identity
     \begin{equation}\label{IdSymp}
     X(t)^TJX(t)=J \;,
     \end{equation}
     i.e. $X(t)$  is $J$-orthogonal or $J$-symplectic. These matrices were studied in
     \cite{Dos-Sad_13,dos-cou-sam_13,god_92,god_89,god-sad_01}. We recall
     that the spectrums of the symplectic matrices are generaly divided into three groups of
eigenvalues (see e.g \cite{Dos-Sad_13,god-sad_01}) :
$N_{\infty}$ eigenvalues outside the unite circle, $N_0=N_{\infty}$
eigenvalues inside the unite circle and $2(N-N_0)$ eigenvalues on
the unite circle.

Considering the symmetric matrix
 \cite{Dos-Sad_13,god_92,god_89,god-sad_01,god-sad_06}
$$
S_0=(1/2)\left((JW)+(JW)^T\right),
$$
where $W$  is a $J$-symplectic  matrix of $\Rset^{2N\times 2N}$  and $J$
a skew-symmetric matrix of  system  (\ref{Eq1}).
S.K. Godunov and  Sadkane in \cite{god-sad_01} gave a classification
of the  eigenvalues which lie on the unit circle  as follows

\begin{definition}\label{d1}
An eigenvalue $\rho$ of $W$ on the unit circle is an eigenvalue of
red color or $r$-eigenvalue (respectively an eigenvalue of green color or $g$-eigenvalue) if
$\left(S_0x,\,x\right)>0$ (respectively $\left(S_0x,\,x\right)<0$)
for any eigenvector $x$ associated with $\rho$. However if
$\left(S_0x,\,x\right)=0,$ then $\rho$ is of mixed color.
\end{definition}
 From this definition, we give the following theorem \cite{Dos-Sad_13}
\begin{thm}\label{t00}
The matrix $W$ is strongly stable if and  only if, one of
the following conditions is  verified
\begin{enumerate}
\item  $W$ has only $r$- and/or $g$- eigenvalues and  the quantity
\begin{equation}\label{dS}
\delta_{S}=min\{|e^{i\theta_l}-e^{i\theta_j}|:e^{i\theta_l},e^{i\theta_j}
\; \text{are\, $r$-\, and \,$g$-eigenvalues\, of\,
$W$}\}
\end{equation}
 should not be close to zero.
\item  $\Pset_r+\Pset_g=I,$ and $\Pset^T_rS_0\Pset_g=0,$
where $\Pset_r$ and $\Pset_g$ are the projectors associated respectively with
$r$-eigenvalues  and $g$-eigenvalues of $W,$ and
$$ S_r=\Pset^T_rS_0\Pset_r=S^T_r \text{ and\;} S_g=\Pset^T_gS_0\Pset_g=S^T_g.$$
 \item  the sequence of averaged  matrix
  $\left(S^{(n)}\right)_{n\ge 0}$  defined by
 ${\displaystyle S^{(n)}=\dfrac{1}{2^n}\sum_{k=1}^{2^n}(W^T)^{k-1}W^{k-1}}$   converges
  to a positive definite symmetric  constant  matrix  $S^{(\infty)}$ and the
 quantity  defined in (\ref{dS})
   is not close to  zero.
\end{enumerate}
\end{thm}

   Regarding the strong stability  analysis  of the Hamiltonian system
     with periodic coefficients, we give the following theorem
     (see \cite{Dos-Sad_13,dos-cou_14,dos-cou-sam_13})

\begin{thm}\label{t1}
System (\ref{Eq2}) is strongly stable if and only if, one of the following
conditions is  verified
\begin{enumerate}
\item[1)] If there exists $\varepsilon>0$ such that any Hamiltonian
system with $P$-periodic coefficients of the form
\begin{equation}\label{Eq3}
    J\dfrac{dX(t)}{dt}=\widetilde{H}(t)X(t),\;
\end{equation}
and verifying
 ${\displaystyle
\|H-\widetilde{H}\|\equiv
\int_{0}^{P}|H(t)-\widetilde{H}(t)|dt<\varepsilon
}$  is stable.
\item[2)] The monodromy matrix $X(P)$ of system (\ref{Eq2}) is
strongly stable.
\end{enumerate}
\end{thm}

Thus the analysis of the strong stability of a Hamiltonian system with
periodic coefficients
is linked to the stability of any small perturbation of the
 system preserving its structure.
 Which leads us to study the perturbation  of these type of system.
 In this paper, we are interesting  in a type of perturbation called
  perturbation of rank $k\ge  2$ of Hamiltonian system with periodic coefficients.
  A  study  of   rank one perturbations was made in \cite{DAK_16}
  from a study of  rank one perturbation of symplectic in \cite{Meh,Meh1}.
  In  our study, we use matrices whose columns generate Lagrangian invariant subspaces.
  Thus  to understand the  rank $k\ge 2$ perturbation theory of
  Hamiltonian systems with periodic coefficients,
  we give some basic properties of the isotropic subspaces in  section \ref{Sec1}.
   in section  \ref{Sec2} the theory of rank $k\ge 2$ perturbations of symplectic matrices is proposed.
 Section \ref{Sec3} explains the concept of rank $k$ perturbation of Hamiltonian systems with periodic coefficients.
  In section \ref{Sec4}, we analyse the Jordan canonical form of  matrizant of rank $k$ perturbation
  of \ref{Eq2}.
     In section \ref{Sec5},  We give some numerical examples which illustrate our theoretical results.
      Finally, we make some concluding remarks in section \ref{Sec6}.\\
      Throughout the paper, we use the following notation:
       The identity and zero matrices of order $k$ are respectively denoted by $I_k,\;\text{and\;} 0_k$
       or just $I$ and $0$ when the order is clear from the context. And by the symbols $\|A\|$
        and $U^T$ we denote the $2$-norm of the matrix $A$ and the transposed matrix (or vector) $U$ respectively.

\section{Some basic notions on some types of subspaces}\label{Sec1}
Start by basic notions on the Lagragian and isotropic subspaces.
\subsection{Lagrangian    subspaces}

These subspaces  are defined as follow \cite{Meh1}
   \begin{definition}\label{d0001}
   Let $J\in \Rset^{2N\times 2N}$  be either skew-symmetric and
    invertible (or in the complex case  only, Hermitian and invertible, respectively).\\
   A  subspace $\mathcal{L}$ of $\Cset^{2N}$ is called $J$-Lagrangian
  if it has the dimension $N$ and
$$
<Jx,y>=0,\quad \forall\, x,y\in \mathcal{L}.
$$
or in the case   $J$  Hermitian   if  $<x,y>_{*}=0,\quad \forall \; x,y\in \mathcal{L}$
where  the standard bilinear  and sesquilinear  forms  are defined as follow
\begin{align*}
<x,y>=\sum_{j=1}^{2N}x_jy_j,\quad x=[x_1,\ldots,x_{2N}]^T,\;y=[y_1,\ldots,y_{2N}]^T\in \Rset^{2N},\\
<x,y>_{*}=\sum_{j=1}^{2N}x_j\overline{y_j},\quad x=[x_1,\ldots,x_{2N}]^T,\;y=[y_1,\ldots,y_{2N}]^T\in \Cset^{2N},
\end{align*}
\end{definition}

Specially, a subspace $\mathcal{L}$ is called Lagrangian subspace if
and only if there exists a matrix $L$ whose columns generating $\mathcal{L}$ satisfies
$rank(L)=N$ and $L^{\star}J_{2N}L=0.$

Consider the following definition
\begin{definition}
A matrix $\mathcal{H}\in \Cset^{2N\times 2N}$  is called Hamiltonian  if
$J\mathcal{H}=\left(J\mathcal{H}\right)^{*}$  is Hermitian, where $J=\left[\begin{array}{cc}
0& I_N\\-I_N & 0\end{array}\right]$  and the superscript  $*$  denotes the conjugate  transpose.
\end{definition}
The following lemma
 gives the link between the $J$-symplectic  and the $J$-Hamiltonian matrices
  via the Caley transform, see e.g., \cite{goh-lan-rod_05,lan-rod_95,Meh1}
  $$
  \mathcal{C}_1(M)=(I-M)^{-1}(I+M),\quad \mathcal{C}_{-1}(N)=(I+N)^{-1}(I-N)
  $$
  for $M,N\in \Rset^{2N\times 2N}$  with $1$ and $-1$ not belonging  to the spectrum of  $M$  and
  $N$  respectively.
\begin{lem}[Caley  Transform]  Let  $W\in \Rset^{2N\times 2N}$  be $J$-symplectic.
\begin{enumerate}
  \item[(i)] If $W$ has not of eigenvalues 1, then the matrix $A=\mathcal{C}_1(W)=(I-W)^{-1}(I+W)$
   is $J$-Hamiltonian  and $\pm 1$  are not eigenvalues of $A$. Moreover, we  have
      $$
      W=\mathcal{C}_1^{-1}(A)=(A-I)(A+I)^{-1}.
      $$
  \item[(ii)] If   $W$   has not of eigenvalues $-1$,  then the matrix  $B=\mathcal{C}_{-1}(W)=
  (I+W)^{-1}(I-W)$  is $J$-Hamiltonian  and $\pm 1$  are  not  eigenvalues  of $B$. Moreover, we have
  $$
  W=\mathcal{C}_{-1}^{-1}(B)=(I-B)(B+I)^{-1}.
  $$
\end{enumerate}

\end{lem}

The following proposition gives us a relations between
 the Lagrangian subspaces and the symplectic matrices (see in \cite{fre-meh-xu_02,Meh1})
\begin{prop}\label{pa01}
\begin{enumerate}
\item
 Let $ W\in\Cset^{2N\times 2N}$ be a symplectic matrix. Then the columns of $W\left[
\begin{array}{c}
I_N\\
0_N
\end{array}
\right]$ span a Lagrangian subspace. Moreover, if the columns
of a  matrix  $L\in\Cset^{2N\times N}$ span a Lagrangian subspace,
 then there exists a symplectic matrix
$\widetilde{W}$ such that $range\left(\widetilde{W}\left[\begin{array}{l}
I_N\\
0_N
\end{array}
\right]\right)=range(L).$
\item
 Let $\mathcal{H}\in \Cset^{2N\times 2N}$  be a Hamiltonian  matrix.
  There  exists  a Lagragian  invariant  subspace $\mathcal{L}$  of
  $\mathcal{H}$  if and only  if there exists  a symplectic matrix $W$
  such  that range $\left(W\left[\begin{array}{c}I_N\\0\end{array}\right]\right)=\mathcal{L}$  and
  we have the Hamiltonian block triangular form
  $$
  W^{-1}\mathcal{H}W=\left[\begin{array}{cc}
  R & D\\0 & -R^*\end{array}\right].
  $$
\end{enumerate}
\end{prop}

\subsection{Isotropic   subspaces}
The isotropic subspaces of certain types of matrices are usually of interest in applications \cite{kres_05,Wat07}.
\begin{definition}
A subspace $\mathcal{X}\subseteq\Rset^{2N}$ is called isotropic if $\mathcal{X}\perp J_{2N}\mathcal{X}$.
A maximal isotropic subspace is called Lagrangian.
\end{definition}

We collect some properties on the isotropic subspaces in the theorem below

\begin{prop}\label{THI}
\begin{enumerate}
\item Let $\mathcal{X}$  be an isotropic subspace.
 Then the dimension of  $\mathcal{X}$  is less than  or equal to  $N$.
  \item  All isotropic  subspace  is contained in a Lagrangian subspaces.
  \item Let $S=[S_1\; S_2]\in \Rset^{2N\times 2N}$ be a symplectic matrix  with
   $S_i\in \Rset^{2N\times N},\;i=1,2$ ;  then the columns of $S_1$  and $S_2$
    span isotropic subspaces.
\end{enumerate}
\end{prop}
Recall us two usefull lemmas on the isotropic subspace \cite{kres_05}
\begin{lem}
 Let $\mathcal{X}_S\subseteq \Rset^{2N}$  be a subspace  that is invariant under  a Hamiltonian matrix $S$
 which has all its eigenvalues associated with $\mathcal{X}_S$  satisfying $\mathcal{R}(\lambda)<0$.  Then
 $\mathcal{X}_S$  is isotropic.
\end{lem}
The below  lemma gives a link between the invariant isotropic subspaces  and the existence of
the  orthogonal symplectic matrices  i.e. the matrix $U$  which has the representation
 $U=\left(\begin{array}{cc}
   U_1 & U_2 \\
    -U_2 & U_1
     \end{array}\right),\;\; U_1,U_2\in \Rset^{n,n}$ \cite{kres_05}.

\begin{lem}
Let $S\in \Rset^{2n\times 2n}$  be a skew-Hamiltonian matrix  and
$X\in \Rset^{2n\times k}(k\leq n)$  with orthogonal  columns. Then
the columns  of $X$  span  an isotropic invariant  subspace  of $S$  if and only if
there exists  an orthogonal  symplectic  matrix
$U=[X,Z,J^TX,J^TZ]$  with  some  $Z\in \Rset^{2n\times (n-k)}$  so that
$$
U^TSU=\left.\begin{array}{cc}
              &  \left.\begin{array}{cccc}
                         k  & n-k & k &  n-k
                       \end{array}\right. \\
            \left.\begin{array}{c}
                    k \\
                    n-k \\
                    k \\
                    n-k
                  \end{array}\right.  &\left[\begin{array}{cccc}
                       A_{11} & A_{12} & G_{11} & G_{12} \\
                       0 & A_{22} & -G_{12}^T & G_{22} \\
                       0 & 0 & A_{11}^T & 0 \\
                       0 & H_{22} & A_{12}^T & A_{22}^T
                     \end{array}\right]
            \end{array}\right.
$$
\end{lem}

We can build isotropic subspaces from the methods of Krylov  subspace. Recall that
the Krylov  subspaces  are  of  the form
$$
\mathcal{K}_m\equiv\mathcal{K}(A,v)=span\left\{v,A,Av,A^2v,\ldots,A^{m-1}v \right\}
$$
where $A\in \Rset^{n,m}$  and $v\in \Rset^m$. The Krylov subspace methods are : the Hermitian or skew-hermitian
Lanczos algorithm  and Arnoldi's method  and its variations. We give the  following  proposition  which contains
some  properties of these subspaces (see \cite[p. 126]{Saad_11})
   \begin{prop}
   \begin{enumerate}
     \item The Krylov subspace $\mathcal{K}_m$ is the subspace  of all vectors in $\mathbb{C}^n$  which  can  be
     written as $x=p(A)v$, where $p$  is a polynomial of degree less than  or equal to $m-1$.
     \item Let $m_0$  be the degree  of the minimal  polynomial  of  $v$.  Then $\mathcal{K}_{m_0}$  is invariant  under
     $A$  and $\mathcal{K}_m=\mathcal{K}_{m_0}$  for all $m\ge m_0$.
     \item The Krylov  subspace  $\mathcal{K}_m$  is of dimension  $m$  if and only if
     the grade  of $v$  with  respect  to $A$  is larger  than  $m-1$.
   \end{enumerate}
   \end{prop}

 Thus any Krylov process constructed from a  skew-Hamiltonian matrix
  automatically produces an isotropic subspace.
  Hence the following proposition (see \cite[p. 399]{Wat07})

\begin{prop}
Let $S\in \Rset^{2N\times 2N}$  be a skew-Hamiltonian matrix
 and $u\in\Rset^{2N}$ be an arbitrary nonzero vector.
 Then the Krylov subspace  $\mathcal{K}_j(S,u)$  is isotropic  for all $j$.
\end{prop}


        \section{Rank $k$ perturbation of symplectic matrices}\label{Sec2}

Consider a symplectic matrix $W$ and a  $J$-Lagrangian subspace
$\mathcal{L}$ of dimension $N$.  Let $u_1,\cdots,u_k$ be  $k$
vectors of  $\mathcal{L}$, where $k\leq N$.
 Setting $U=[u_1;\ldots;u_k],$ and considering the
matrix
 $$\widetilde{W}=\left(I+UU^{T}J\right)W, $$
  we have the following proposition
\begin{prop}
  the matrix $\widetilde{W}$ is $J$-symplectic.
\end{prop}
\begin{proof} We have the  following inequalities
 \begin{align*}
 \widetilde{W}^TJ\widetilde{W}&=W^T(I-JUU^T)(J+JUU^TJ)W\\
 &=W^T\left[J+\underbrace{JUU^TJ-JUU^TJ}_{=0}-JU\underbrace{(U^TJU)}_{=0}U^TJ\right]W\\
 &=W^TJW=J.
 \end{align*}
\end{proof}
 The following proposition is a set of results deduced  from  \cite{Yan2}.
\begin{prop}\label{p12} 
Consider the matrix  $\widetilde{I}=(I+UU^TJ)$. Then
\begin{enumerate}
\item[1)] $\widetilde{I}$ is  $J$-symplectic.
\item[2)]$\widetilde{I}^{-1}=I-UU^{T}J.$
\item[3)] $ dim\left(ker(\widetilde{I}-I)\right)=2N-k,$ where $k$ is the
rank of $U.$
\item[4)] $1\in \sigma(\widetilde{I}),$ where $\sigma(\widetilde{I})$ is the spectrum of
$\widetilde{I}.$ 
\end{enumerate}
\end{prop}
\begin{proof}
The proof is easily deduced from those  of \cite{Yan2}.
\end{proof}
From the foregoing, we give the following definition
\begin{definition} \label{d}
Let $W$ be a symplectic matrix. We call rank $k$
perturbation of $W,$ any matrix of the form
\begin{equation}\label{er1}
\widetilde{W}=(I+UU^TJ)W,
\end{equation}
where $U$ is a matrix of rank $k$ whose columns belong in
a $J$-Lagrangian subspace.
\end{definition}
\begin{rem}\label{r1}
 The matrix $\widetilde{W}$   can be put in the form
$$
\widetilde{W}=(I+\sum_{j=1}^{k}u_ju_j^TJ)W.
$$
\end{rem}
More specially, this remark shows that any rank $k$ perturbation of $W$ is $k$
rank one perturbations of the symplectic matrix $W.$ We have
$$
\left(\prod_{j=1}^{k}\left(I+u_ju^T_jJ\right)\right)W=\left( I+\sum_{j=1}^{k}u_ju^T_jJ \right)W.
$$

Consider a  symplectic  matrix of function  $(X(t))_{t\in\Rset}$ ;
we can consider for example the solution of Hamiltonian  system (\ref{Eq2})
which are $J$-symplectic.
 We have the following definition
\begin{definition}\label{d0}
We call rank $k$ perturbation of $X(t)$ any  function matrix  of the form
\begin{equation}\label{aq}
\widetilde{X}(t)=(I+UU^TJ)X(t),
\end{equation}
where $rank(U)=k$ and the columns of $U$ belong in a $J$-Lagrangian
subspace.
\end{definition}
\begin{rem}\label{r2}
Since the  function  matrix  $(X(t))_{t\in \Rset}$ is $J$-symplectic, its rank $k$ perturbation will be $J-$symplectic.
\end{rem}
From definition \ref{d0} and remark \ref{r2}, we can introduce the theory of rank
$k$ perturbation of Hamiltonian system with periodic coefficients.

\section{Rank $k$ perturbation of Hamiltonian system with periodic
coefficients}\label{Sec3}

Let $U$ be a constant matrix of rank $k$ such that its  columns
belong in a $J$-Lagrangian subspace and $(X(t))_{t\geq 0}$ be the
fundamental solution of (\ref{Eq2}).
We have the
following proposition
\begin{prop}\label{p1}
Consider the following perturbed Hamiltonian system
\begin{equation}\label{er3}
   J\dfrac{d\widetilde{X}(t)}{dt} =\left[H(t)+E(t)\right]\widetilde{X}(t)
\end{equation}
where $$
E(t)=(JUU^{T}H(t))^{T}+JUU^{T}H(t)+(UU^{T}J)^{T}H(t)(UU^{T}J).
$$Then $\widetilde{X}(t)=(I+UU^{T}J)X(t)$ is a solution of system
(\ref{er3}).
\end{prop}

\begin{proof}~~\\
 By derivation of $\widetilde{X}(t)$, we obtain:
\begin{align*}
    J\dfrac{d\widetilde{X}(t)}{dt}=&J(I+UU^{T}J)J^{-1}J\dfrac{dX(t)}{dt}\\
    =&J(I+UU^{T}J)J^{-1}H(t)X(t),\quad \text{according \,form \, system}\,(\ref{Eq2})\\
    =&[H(t)+JUU^{T}H(t)](I+UU^TJ)^{-1}\widetilde{X}(t)\\
    =&\left[H(t)+\underbrace{(JUU^{T}H(t))^{T}+JUU^{T}H(t)+(UU^{T}J)^TH(t)(UU^{T}J)}_{E(t)}\right]\widetilde{X}(t)
 \end{align*}
 Hence system (\ref{er3})
 where
 \begin{equation}\label{elia}
E(t)=(JUU^{T}H(t))^{T}+JUU^{T}H(t)+(UU^{T}J)^{T}H(t)(UU^{T}J).
\end{equation}
\end{proof}
We can easily check that $E(t)$ is symmetric and periodic i.e.
$E(t)^T=E(t)$ and $E(t+P)=E(t)$ for all $t\in \Rset_{+}.$ \\ The
following corollary gives us a simplified form of system (\ref{er3}), with $X(0)=I.$
\begin{cor}\label{c2} Equation (\ref{er3}) can be put in  the
form
\begin{equation}\label{er4}
\left\{\begin{array}{rcl}
    J\dfrac{d\widetilde{X}(t)}{dt}&=&\left(I-UU^TJ\right)^TH(t)\left(I-UU^TJ\right)\widetilde{X}(t),\; t\in \Rset_{+},\\
    & & \\
     \widetilde{X}(0)&=&I+UU^TJ
     \end{array}\right.
\end{equation}
\end{cor}
\begin{proof}
Developing $(I-UU^{T}J)^{T}H(t)(I-UU^{T}J),$ we see that
\begin{align*}
(I-UU^{T}J)^{T}H(t)(I-UU^{T}J)&=H(t)+\\
&\underbrace{(J^{T}UU^{T}H(t))^{T}+J^{T}UU^{T}H(t)+(UU^{T}J)^{T}H(t)(UU^{T}J)}_{E(t)}
\end{align*}
and $\widetilde{X}(0)=(I+UU^TJ)X(0)=I+UU^TJ.$
\end{proof}
We give the following corollary
\begin{cor}\label{c3}
Any solution $(\widetilde{X}(t))_{t\geq 0}$ of the perturbed system
(\ref{er3}) of system (\ref{Eq2}), is of the form
$$
\widetilde{X}(t)=(I+UU^TJ)X(t),
$$ where $(X(t))_{t\geq 0}$ is the
fundamental solution of system (\ref{Eq2}).
\end{cor}
\begin{proof}
From proposition \ref{p1}, if $X(t)$ is the solution of (\ref{Eq2}),
then the perturbed matrix $\widetilde{X}(t)=(I+UU^TJ)X(t)$ is the solution
of (\ref{er4}). Reciprocally, for any solution $\widetilde{X}(t)$ of
(\ref{er4}), let
$$
X(t)=(I-UU^TJ)\widetilde{X}(t)
$$ where $U$ is the matrix defined in system (\ref{er4}). Then
$\widetilde{X}(t)=(I+UU^TJ)X(t).$ Replacing $\widetilde{X}(t)$ in
(\ref{er4}), we get
\begin{align*}
J(I+UU^TJ)\dfrac{dX(t)}{dt}&=(I-UU^TJ)^TH(t)X(t)\\
\underbrace{(I+UU^TJ)^TJ(I+UU^TJ)}_{=J}\dfrac{dX(t)}{dt}&=H(t)X(t)\\
\Rightarrow J\dfrac{dX(t)}{dt}&=H(t)X(t)
\end{align*}
and $X(0)=(I-UU^TJ)\widetilde{X}(0)=(I-UU^TJ)(I+UU^TJ)=I.$
Consequently, $X(t)$ is the solution of (\ref{Eq2}).
\end{proof}

\begin{rem}\label{r3}
Basing on remark \ref{r1}, system (\ref{er4}) can be
written as below
\begin{equation}\label{A}
\left\{\begin{array}{rcl}
    J\dfrac{d\widetilde{X}(t)}{dt}&=&\left(I-\sum_{j=1}^{k}u_ju^T_jJ\right)^{T}H(t)\left(I-\sum_{j=1}^{k}u_ju^T_jJ\right)\widetilde{X}(t)\\
    & & \\
     \widetilde{X}(0)&=&(I+\sum_{j=1}^{k}u_ju^T_jJ)
     \end{array}\right.
\end{equation}
where each vector $\left(u_j\right)_{1\leq j\leq k}\subset\Rset^{2N}$ belongs in a same $J$-Lagrangian subspace.
\end{rem}
 We can immediately see that the rank $k$
perturbation of (\ref{Eq2}) can be interpreted as $k$ rank one perturbations of (\ref{Eq2}). In fact, since
$$
I-UU^TJ=I-\sum_{j=1}^{k}u_ju^T_jJ=\prod_{j=1}^{k}\left(I-u_ju^T_jJ\right),
$$ we easily see that  system (\ref{A}) can be put in  the following form
\begin{equation}\label{A1}
\left\{\begin{array}{rcl}
    J\dfrac{d\widetilde{X}(t)}{dt}&=&\left(\prod_{j=1}^{k}\left(I-u_ju^T_jJ\right)\right)^{T}H(t)\left(\prod_{j=1}^{k}\left(I-u_ju^T_jJ\right)\right)\widetilde{X}(t)\\
    & & \\
     \widetilde{X}(0)&=&\prod_{j=1}^{k}\left(I+u_ju^T_jJ\right)
     \end{array}\right.
\end{equation}
which is the same as the bellow system, for all $p\in \{ 1,2,...,k-1\}:$

\begin{equation}\label{A3}
\left\{\begin{array}{rcl}
    J\dfrac{d\widetilde{X}(t)}{dt}&=&\left(\prod_{j=p+1}^{k}(I-u_ju^T_jJ)\right)^{T}H^{(p)}(t)\left(\prod_{j=p+1}^{k}(I-u_ju^T_jJ)\right)\widetilde{X}(t)\\
    & & \\
     \widetilde{X}(0)&=&\left(\prod_{j=p+1}^{k}(I+u_{(k+p-j+1)}u^T_{(k+p-j+1)}J)\right)\overline{X}^{(p)}(0)
     \end{array}\right.
\end{equation}
where $$
H^{(p)}(t)=\left(\prod_{j=1}^{p}(I-u_ju^T_jJ)\right)^{T}H(t)\left(\prod_{j=1}^{p}(I-u_ju^T_jJ)\right)\;\text{ and\;} \overline{X}^{(p)}(0)=\prod_{j=1}^{p}\left(I+u_{(p-j+1)}u^T_{(p-j+1)}J\right).
$$
Now, let us interest to the Jordan canonical form  of the solution $(\widetilde{X}(t))_{t\geq 0}$ of the perturbed system (\ref{er4}) of (\ref{Eq2}) in following section.

\section{Jordan canonical form  of    $(\widetilde{X}(t))_{t>0}$ }\label{Sec4}

\begin{thm}\label{THR}
Let $J\in \Cset^{2N\times 2N}$ be  skew-symmetric and nonsingular   matrix,
 $(X(t))_{t>0}$ fondamental solution of system (\ref{Eq2})   and $\lambda(t)\in \Cset$
  an  eigenvalue of  $X(t)$ for all $t>0$.
  Assume  that   $X(t)$  has  the   Jordan  canonical  form
   $$
  \left(\bigoplus_{j=1}^{l_1} \mathcal{J}_{n_1}(\lambda(t))\right)\oplus
  \left(\bigoplus_{j=1}^{l_2}\mathcal{J}_{n_2}(\lambda(t))\right)
  \oplus\cdots\oplus\left(
  \bigoplus_{j=1}^{l_{m(t)}}\mathcal{J}_{n_{m(t)}}(\lambda(t))\right)\oplus \mathcal{J}(t),
  $$
where  $n_1>\cdots >n_{m(t)}$  with $m\; :\; \Rset \longrightarrow \Nset^*$
 a  function  of index such that the algebraic multiplicities is $a(t)=l_1n_1+\cdots+l_{m(t)}n_{m(t)}$
   and    $\mathcal{J}(t)$  with  $\sigma(\mathcal{J}(t))\subseteq \Cset\setminus\{\lambda(t)\}$
  contains all  Jordan blocks associated with  eigenvalues  different  from  $\lambda(t)$.
   Furthermore, let   $B(t)=UU^TJX(t)$ where $U\in \Cset^{2N\times k}$ is such that its
   columns generate a Lagrangian subspace.

\begin{enumerate}
  \item[(1)]  If   $\forall t>0$, $\lambda(t)\not\in \{-1,1\}$, then  generically with respect   to the components  of $U$,
    the matrix  $X(t)+B(t)$  has the Jordan canonical  form

       $$\left\{\begin{array}{lcl}
       {\displaystyle \left(\bigoplus_{j=1}^{l_1-k} \mathcal{J}_{n_1}(\lambda(t))\right)\oplus
  \left(\bigoplus_{j=1}^{l_2}\mathcal{J}_{n_2}(\lambda(t))\right)\oplus\cdots
      \oplus
      \left(
  \bigoplus_{j=1}^{l_{m(t)}}\mathcal{J}_{n_{m(t)}}(\lambda(t))\right)\oplus \widetilde{\mathcal{J}}(t)},& \text{if} & k\leq l_1
  \\
  & & \\
  {\displaystyle \left(\bigoplus_{j=1}^{l_i-k_i} \mathcal{J}_{n_i}(\lambda(t))\right)\oplus
  \left(\bigoplus_{j=1}^{l_{i+1}}\mathcal{J}_{n_{i+1}}(\lambda(t))\right)\oplus\cdots
      \oplus
      \left(
  \bigoplus_{j=1}^{l_{m(t)}}\mathcal{J}_{n_{m(t)}}(\lambda(t))\right)\oplus \widetilde{\mathcal{J}}(t)},& \text{if} &
  \left\{\begin{array}{c}{\displaystyle k=\sum_{s=1}^{i-1}
  l_s+k_i}\\\text{with}\;\; k_i\leq l_i\\ \text{and}\;\; i>1.\end{array}\right.
  \end{array}\right.
  $$

      where   $\widetilde{\mathcal{J}}(t)$   contains  all the Jordan  blocks of $X(t)+B(t)$   associated
      with  eigenvalues  different  from    $\lambda(t)$.

  \item[(2)]  If  $\exists  t_0 >0$, verifying    $\lambda(t_0)\in \{+1,1\}$, we  have
   \begin{enumerate}
   \item[(2a)]  if  ${\displaystyle k=\sum_{s=1}^{i-1}l_s+k_i}$ with $n_1,n_2,\ldots,n_i$  are even and  $k_i\leq l_i$,
     then generically  with   respect to   the components of $U$,  the matrix
     $X(t_0)+B(t_0)$   has  the Jordan  canonical form
      $$
      \left(\bigoplus_{j=1}^{l_i-k_i} \mathcal{J}_{n_i}(\lambda(t_0))\right)\oplus
  \left(\bigoplus_{j=1}^{l_{i+1}}\mathcal{J}_{n_{i+1}}(\lambda(t_0))\right)
      \oplus\cdots\oplus
      \left(
  \bigoplus_{j=1}^{l_{m(t)}}\mathcal{J}_{n_{m(t)}}(\lambda(t_0))\right)\oplus \widetilde{\mathcal{J}}(t_0),
      $$

      where   $\widetilde{\mathcal{J}}(t_0)$   contains all   the Jordan blocks  of   $X(t_0)+B(t_0)$
      associated with eigenvalues  different  from   $\lambda(t_0)$.
 \item[(2b)]  if ${\displaystyle  k=\sum_{s=1}^{i-1}l_s+2k_i-1}$ with $2k_i\leq l_i$ and  $n_i$  is odd, then  $l_i$
  is even   and  generically  with   respect   to   the components of $U$,  the matrix
     $X(t_0)+B(t_0)$   has the Jordan  canonical form

$$
     \mathcal{J}_{n_i+1}(\lambda(t_0)) \oplus\left(\bigoplus_{j=1}^{l_i-2i} \mathcal{J}_{n_i}(\lambda(t_0))\right)
     \oplus\cdots\oplus
     \left(
  \bigoplus_{j=1}^{l_{m(t)}}\mathcal{J}_{n_{m(t)}}(\lambda(t_0))\right)\oplus \widetilde{\mathcal{J}}(t_0),
  $$

      where  $\widetilde{\mathcal{J}}(t_0)$   contains all the Jordan  blocks of  $X(t_0)+B(t_0)$
       associated with eigenvalues different   from      $\lambda(t_0)$.
       \end{enumerate}
\end{enumerate}
\end{thm}

\begin{proof}
we recall that the rank $k$  perturbation  $X(t)+B(t)$   of $X(t)$  can be put on the form
of  $k$  rank one  perturbation   $(X(t))_{t>0}$ by
$$
\widetilde{X}(t)=\left[ \prod_{j=1}^{k}\left(I+u_{k-j+1} u_{k-j+1}^TJ \right)\right]X(t)
$$
where each vector $u_j$ are  the columns of the matrix $U$.
\begin{enumerate}
  \item  If $\lambda(t)\not\in \left\{ -1,1\right\}$,
   $\forall  t\ge 0$,
  \begin{itemize}
    \item For $k\leq l_1$, we have (see \cite[Theorem 10]{DAK_16} ) :
      \begin{itemize}
   \item  $ \widetilde{X}_1= \left(I+u_1u_1^T\right)X(t)$  has the   following  Jordan  canonical form
        $$
  \left(\bigoplus_{j=1}^{l_1-1} \mathcal{J}_{n_1}(\lambda(t))\right)\oplus
  \left(\bigoplus_{j=1}^{l_2}\mathcal{J}_{n_2}(\lambda(t))\right)
  \oplus\cdots\oplus\left(
  \bigoplus_{j=1}^{l_{m(t)}}\mathcal{J}_{n_{m(t)}}(\lambda(t))\right)\oplus \widetilde{\mathcal{J}}_1(t),
  $$
    where  $ \widetilde{\mathcal{J}}_1(t)$ contains all  the Jordan blocks of $\widetilde{X}_1(t)$
    associated with  eigenvalues different from $\lambda(t)$.
    \item $ \widetilde{X}_2=\prod_{j=1}^{2}\left(I+u_{2-j+1}u_{2-j+1}^T\right)=\left(I+u_2u_2^T\right)(\left(I+u_1u_1^T\right)X(t)$  has
    the following  Jordan canonical form
      $$
  \left(\bigoplus_{j=1}^{l_1-2} \mathcal{J}_{n_1}(\lambda(t))\right)\oplus
  \left(\bigoplus_{j=1}^{l_2}\mathcal{J}_{n_2}(\lambda(t))\right)
  \oplus\cdots\oplus\left(
  \bigoplus_{j=1}^{l_{m(t)}}\mathcal{J}_{n_{m(t)}}(\lambda(t))\right)\oplus \widetilde{\mathcal{J}}_2(t),
  $$
  where  $ \widetilde{\mathcal{J}}_2(t)$ contains all  the Jordan blocks of $\widetilde{X}_2(t)$
    associated with  eigenvalues different from $\lambda(t)$.
    \item $ \widetilde{X}(t)=\widetilde{X}_k=\prod_{j=1}^{k}\left(I+u_{k-j+1}u_{k-j+1}^T\right)$ has
     the following  Jordan canonical form
      $$
  \left(\bigoplus_{j=1}^{l_1-k} \mathcal{J}_{n_1}(\lambda(t))\right)\oplus
  \left(\bigoplus_{j=1}^{l_2}\mathcal{J}_{n_2}(\lambda(t))\right)
  \oplus\cdots\oplus\left(
  \bigoplus_{j=1}^{l_{m(t)}}\mathcal{J}_{n_{m(t)}}(\lambda(t))\right)\oplus \widetilde{\mathcal{J}}_k(t),
  $$
  where  $\widetilde{\mathcal{J}}(t)=\widetilde{\mathcal{J}}_k(t)$  contains all   the Jordan blocks of $\widetilde{X}(t)$
    associated with  eigenvalues different from $\lambda(t)$.
    \end{itemize}
    \item  For $k=\sum_{s=1}^{i-1}l_s+k_i$  with $k_i\leq l_i$ ;
    \begin{itemize}
      \item if  $i=2$, then  $k=l_1+k_i$.  We have
      $$\widetilde{X}(t)=  \left[\prod_{j=1}^{k-l_1}\left(I+u_{k-j+1}u_{k-j+1}^T\right)\right].
     \underbrace{\left[\prod_{j=k-l_1+1}^{k}\left(I+u_{k-j+1}u_{k-j+1}^T\right)\right]X(t) }_{\widetilde{X}_{l_1}}$$
     where $\widetilde{X}_{l_1}(t)$  is $l_1$ rank one perturbations of $X(t)$ ; then the symplectic matrix
       $\widetilde{X}_{l_1}(t)$  therefore  has   the following  Jordan  canonical    form
     $$ \left(\bigoplus_{j=1}^{l_2} \mathcal{J}_{n_2}(\lambda(t))\right)\oplus
  \left(\bigoplus_{j=1}^{l_3}\mathcal{J}_{n_3}(\lambda(t))\right)
  \oplus\cdots\oplus\left(
  \bigoplus_{j=1}^{l_{m(t)}}\mathcal{J}_{n_{m(t)}}(\lambda(t))\right)\oplus \widetilde{\mathcal{J}}_{l_1}(t)
  $$
  using \cite[Theorem 10]{DAK_16}. On the other hand
     $\widetilde{X}(t)$ is $k_i$  rank one  perturbations  of $\widetilde{X}_{l_1}(t)$  with $k_1<l_2$ ;  it
  therefore has the following Jordan form
    $$ \left(\bigoplus_{j=1}^{l_2-k_i} \mathcal{J}_{n_2}(\lambda(t))\right)\oplus
  \left(\bigoplus_{j=1}^{l_3}\mathcal{J}_{n_3}(\lambda(t))\right)
  \oplus\cdots\oplus\left(
  \bigoplus_{j=1}^{l_{m(t)}}\mathcal{J}_{n_{m(t)}}(\lambda(t))\right)\oplus \widetilde{\mathcal{J}}_k(t) ;
  $$
      \item if  $i>2$. Putting $ \alpha(i)=\sum_{s=1}^{i-1}l_s$,  we have
       $$\widetilde{X}(t)=  \left[\prod_{j=1}^{k-\alpha(i)}\left(I+u_{k-j+1}u_{k-j+1}^T\right)\right].
     \underbrace{\left[\prod_{j=k-\alpha(i)+1}^{k}\left(I+u_{k-j+1}u_{k-j+1}^T\right)\right]X(t) }_{\widetilde{X}_{\alpha(i)}}$$
     where  $\widetilde{X}_{\alpha(i)}$  is    $\alpha(i)$ rank one perturbations of  $X(t)$.  Using \cite[Theorem 10]{DAK_16},
      the symplectic matrix
      $\widetilde{X}_{\alpha(i)}$ has the following Jordan form
      $$ \left(\bigoplus_{j=1}^{l_i} \mathcal{J}_{n_i}(\lambda(t))\right)\oplus
  \left(\bigoplus_{j=1}^{l_{i+1}}\mathcal{J}_{n_{i+1}}(\lambda(t))\right)
  \oplus\cdots\oplus\left(
  \bigoplus_{j=1}^{l_{m(t)}}\mathcal{J}_{n_{m(t)}}(\lambda(t))\right)\oplus \widetilde{\mathcal{J}}_{\alpha(i)}(t)
  $$
   where  $\widetilde{\mathcal{J}}_{\alpha(i)}(t)$
   contains all the Jordan blocks of $\widetilde{X}_{\alpha(i)}$
    associated with  eigenvalues different from $\lambda(t)$.
      On the other hand
     $\widetilde{X}(t)$ is $k_i$  rank one  perturbations  of $\widetilde{X}_{l_1}(t)$  with $k_1<l_2$ ;  it
  therefore has the following Jordan form
    $$ \left(\bigoplus_{j=1}^{l_i-k_i} \mathcal{J}_{n_i}(\lambda(t))\right)\oplus
  \left(\bigoplus_{j=1}^{l_{i+1}}\mathcal{J}_{n_{i+1}}(\lambda(t))\right)
  \oplus\cdots\oplus\left(
  \bigoplus_{j=1}^{l_{m(t)}}\mathcal{J}_{n_{m(t)}}(\lambda(t))\right)\oplus \widetilde{\mathcal{J}}_k(t)
  $$
   where  $\widetilde{\mathcal{J}}(t)= \widetilde{\mathcal{J}}_k(t)$
   contains all the Jordan blocks of $\widetilde{X}(t)$
    associated with  eigenvalues different from $\lambda(t)$.
    \end{itemize}
  \end{itemize}
  \item Consider that there exists   $t_0>0$ verifying  $\lambda(t_0)\in \{1,-1\}$.
  \begin{itemize}
    \item if  ${\displaystyle k=\sum_{s=1}^{i-1}l_s+k_i}$  with  $n_1,n_2,..., n_i$ are even and $k_i\leq l_i$, then
    using   \cite[Theorem 10, $(2a)$]{DAK_16},  we have : the symplectic matrix  $\widetilde{X}(t)$,
     $k$ rank one  perturbations of $X(t)$,  has the following canonical Jordan  form
         $$ \left(\bigoplus_{j=1}^{l_i-k_i} \mathcal{J}_{n_i}(\lambda(t))\right)\oplus
  \left(\bigoplus_{j=1}^{l_{i+1}}\mathcal{J}_{n_{i+1}}(\lambda(t))\right)
  \oplus\cdots\oplus\left(
  \bigoplus_{j=1}^{l_{m(t)}}\mathcal{J}_{n_{m(t)}}(\lambda(t))\right)\oplus \widetilde{\mathcal{J}}(t)
  $$
   where  $\widetilde{\mathcal{J}}(t) $
   contains all the Jordan blocks of $\widetilde{X}(t)$
    associated with  eigenvalues different from $\lambda(t)$.

    \item if ${\displaystyle  k=\sum_{s=1}^{i-1}l_s+2k_i-1}$ with $2k_i\leq l_i$ and  $n_i$  is odd, then we have
    \begin{itemize}
      \item  for $i=1$,  $k=2k_1-1$  and $n_1$ is  odd. According to $(2b)$ of \cite[Theorem 10]{DAK_16},
      $l_1$  is even  and we  have
      $$
\widetilde{X}(t)=\left[ \prod_{j=1}^{2k_1-1}\left(I+u_{2k_1-j} u_{2k_1-j}^TJ \right)\right]X(t)
$$
and step by step we have
     \begin{itemize}
       \item  $\widetilde{X}_1(t)=\left(I+u_1 u_1^TJ \right)X(t)$  has the following canonical
       Jordan form
       $$
     \mathcal{J}_{n_1+1}^{(1)}(\lambda(t_0)) \oplus\left(\bigoplus_{j=1}^{l_1-2} \mathcal{J}_{n_i}(\lambda(t_0))\right)
     \oplus\cdots\oplus
     \left(
  \bigoplus_{j=1}^{l_{m(t)}}\mathcal{J}_{n_{m(t)}}(\lambda(t_0))\right)\oplus \widetilde{\mathcal{J}}_1(t_0),
  $$

      where  $\widetilde{\mathcal{J}}_1(t_0)$   contains all the Jordan  blocks of  $\widetilde{X}_1(t_0)$
       associated with eigenvalues different   from      $\lambda(t_0)$.
       \item $\widetilde{X}_2(t)= \left(I+u_2 u_2^TJ\right) \left(I+u_1 u_1^TJ\right) X(t)$  has   the following canonical
       Jordan form
       $$
     \left(\bigoplus_{j=1}^{l_1-2} \mathcal{J}_{n_i}(\lambda(t_0))\right)
     \oplus\cdots\oplus
     \left(
  \bigoplus_{j=1}^{l_{m(t)}}\mathcal{J}_{n_{m(t)}}(\lambda(t_0))\right)\oplus \widetilde{\mathcal{J}}_2(t_0),
  $$
  using $(2a)$ of \cite[Theorem 10]{DAK_16}   because $n_1+1$ is even.
       \item $\widetilde{X}_3(t)=\left[ \prod_{j=1}^{3}\left(I+u_{4-j} u_{4-j}^TJ \right)\right]X(t)$  has
        the following canonical
       Jordan form
        $$
     \mathcal{J}_{n_1+1}^{(2)}(\lambda(t_0)) \oplus\left(\bigoplus_{j=1}^{l_1-2\times 2} \mathcal{J}_{n_1}(\lambda(t_0))\right)
     \oplus\cdots\oplus
     \left(
  \bigoplus_{j=1}^{l_{m(t)}}\mathcal{J}_{n_{m(t)}}(\lambda(t_0))\right)\oplus \widetilde{\mathcal{J}}_3(t_0),
  $$
      where  $\widetilde{\mathcal{J}}(t_0)$   contains all the  Jordan  blocks of  $\widetilde{X}_3(t_0)$
       associated with eigenvalues different   from
           $\lambda(t_0)$  using  $(2b)$ of \cite[Theorem 10]{DAK_16}.

       \item $\widetilde{X}(t)=\left[ \prod_{j=1}^{2k_1-1}\left(I+u_{2k_1-j} u_{2k_1-j}^TJ \right)\right]X(t)$ has
        the following canonical  Jordan form
        $$
     \mathcal{J}_{n_1+1}^{(k_1)}(\lambda(t_0)) \oplus\left(\bigoplus_{j=1}^{l_1-2k_1} \mathcal{J}_{n_1}(\lambda(t_0))\right)
     \oplus\cdots\oplus
     \left(
  \bigoplus_{j=1}^{l_{m(t)}}\mathcal{J}_{n_{m(t)}}(\lambda(t_0))\right)\oplus \widetilde{\mathcal{J}}_k(t_0),
  $$
    where  $\widetilde{\mathcal{J}}(t_0)=\widetilde{\mathcal{J}}_k(t_0)$   contains all the Jordan  blocks   of  $\widetilde{X}(t_0)$
       associated with eigenvalues different   from      $\lambda(t_0)$.

     \end{itemize}

      \item for $i=2$, $k=l_1+2k_2-1$   and $n_2$  odd  and we have
       $$\widetilde{X}(t)=  \left[\prod_{j=1}^{k-l_1}\left(I+u_{k-j+1}u_{k-j+1}^T\right)\right].
     \underbrace{\left[\prod_{j=k-l_1+1}^{k}\left(I+u_{k-j+1}u_{k-j+1}^T\right)\right]X(t) }_{\widetilde{X}_{l_1}}$$

      \begin{itemize}
        \item if $n_1$ is even, then using $(2a)$  \cite[Theorem 19]{DAK_16},
         $\widetilde{X}_{l_1}$  has  the following  Jordan canonical  form

     $$
    \left(\bigoplus_{j=1}^{l_2} \mathcal{J}_{n_2}(\lambda(t_0))\right)
     \oplus\cdots\oplus
     \left(
  \bigoplus_{j=1}^{l_{m(t)}}\mathcal{J}_{n_{m(t)}}(\lambda(t_0))\right)\oplus \widetilde{\mathcal{J}}_{l_1}(t_0),
  $$
  and using  the preceding point  $\widetilde{X}(t)$ has the following   Jordan  canonical  form
   $$
     \mathcal{J}_{n_2+1}^{(k_2)}(\lambda(t_0)) \oplus\left(\bigoplus_{j=1}^{l_2-2k_2} \mathcal{J}_{n_2}(\lambda(t_0))\right)
     \oplus\cdots\oplus
     \left(
  \bigoplus_{j=1}^{l_{m(t)}}\mathcal{J}_{n_{m(t)}}(\lambda(t_0))\right)\oplus \widetilde{\mathcal{J}}_k(t_0),
  $$
    where  $\widetilde{\mathcal{J}}(t_0)=\widetilde{\mathcal{J}}_k(t_0)$   contains all the Jordan  blocks of  $\widetilde{X}(t_0)$
       associated with eigenvalues different   from      $\lambda(t_0)$.

        \item if $n_1$  is odd then according $(2b)$ of \cite[Theorem 10]{DAK_16}, $l_1$ is even and
        we deduct that
           $\widetilde{X}_{l_1}$   also   has  the following form
          $$
    \left(\bigoplus_{j=1}^{l_2} \mathcal{J}_{n_2}(\lambda(t_0))\right)
     \oplus\cdots\oplus
     \left(
  \bigoplus_{j=1}^{l_{m(t)}}\mathcal{J}_{n_{m(t)}}(\lambda(t_0))\right)\oplus \widetilde{\mathcal{J}}_{l_1}(t_0),
  $$
   by successively applying $l_1$ rank one perturbations.

       Using again the previous point,  we deduct that
       $\widetilde{X}(t)$ has the  Jordan  canonical   form
   $$
     \mathcal{J}_{n_2+1}^{(k_2)}(\lambda(t_0)) \oplus\left(\bigoplus_{j=1}^{l_2-2k_2} \mathcal{J}_{n_2}(\lambda(t_0))\right)
     \oplus\cdots\oplus
     \left(
  \bigoplus_{j=1}^{l_{m(t)}}\mathcal{J}_{n_{m(t)}}(\lambda(t_0))\right)\oplus \widetilde{\mathcal{J}}_k(t_0),
  $$
    where  $\widetilde{\mathcal{J}}(t_0)=\widetilde{\mathcal{J}}_k(t_0)$   contains all the Jordan  blocks of  $\widetilde{X}(t_0)$
       associated with eigenvalues different   from      $\lambda(t_0)$.

      \end{itemize}
      \item for $i>2$, $n_i$  is odd. Whether $n_1$, $n_2$..... $n_{i-1}$  are  even or odd,
      using successively  $(2a)$ and $(2b)$ of \cite[Theorem 10]{DAK_16}, we deduct that
       $\widetilde{X}_{\alpha(i)}$   also   has  the following  Jordan  canonical  form
          \begin{equation}\label{FE}
    \left(\bigoplus_{j=1}^{l_i} \mathcal{J}_{n_i}(\lambda(t_0))\right)
     \oplus\cdots\oplus
     \left(
  \bigoplus_{j=1}^{l_{m(t)}}\mathcal{J}_{n_{m(t)}}(\lambda(t_0))\right)\oplus \widetilde{\mathcal{J}}_{\alpha(i-1)}(t_0),
  \end{equation}
   by successively applying ${\displaystyle \alpha(i)=\sum_{s=1}^{i-1}l_s}$ rank one perturbations.
   Since $n_i$  is odd, we affirm,  using  $(2b)$ of \cite[Theorem  10]{DAK_16},   that
   $l_i$  is even.
      To end,  using the preceding point,  we  deduct that
         $\widetilde{X}(t)$ has the  canonical Jordan  form
   $$
     \mathcal{J}_{n_i+1}^{(k_i)}(\lambda(t_0)) \oplus\left(\bigoplus_{j=1}^{l_i-2k_i} \mathcal{J}_{n_i}(\lambda(t_0))\right)
     \oplus\cdots\oplus
     \left(
  \bigoplus_{j=1}^{l_{m(t)}}\mathcal{J}_{n_{m(t)}}(\lambda(t_0))\right)\oplus \widetilde{\mathcal{J}}_k(t_0),
  $$
    where  $\widetilde{\mathcal{J}}(t_0)=\widetilde{\mathcal{J}}_k(t_0)$   contains all the Jordan  blocks of  $\widetilde{X}(t_0)$
       associated with eigenvalues different   from      $\lambda(t_0)$.
    \end{itemize}
  \end{itemize}
\end{enumerate}

\end{proof}

\begin{rem}
In   point $(2)$ of  \cite[Theorem 10]{DAK_16}, if ${\displaystyle k=\sum_{s=1}^{i-1}l_s+2k_i}$  with $2k_i\leq l_i$  and  $n_i$  is odd,
then $l_i$ is even  and generally with respect to  the components  of  $U$,  the rank $k$ perturbation
$\widetilde{X}(t_0)=X(t_0)+B(t_0)$  of $X(t_0)$, has
the canonical Jordan  form

$$
\left(\bigoplus_{j=1}^{l_i-2k_i} \mathcal{J}_{n_i}(\lambda(t_0))\right)
     \oplus\cdots\oplus
     \left(
  \bigoplus_{j=1}^{l_{m(t_0)}}\mathcal{J}_{n_{m(t_0)}}(\lambda(t_0))\right)\oplus \widetilde{\mathcal{J}}_k(t_0),
  $$
    where  $\widetilde{\mathcal{J}}(t_0)=\widetilde{\mathcal{J}}_k(t_0)$   contains all the Jordan   blocks of  $\widetilde{X}(t_0)$
       associated with eigenvalues different   from      $\lambda(t_0)$.
\end{rem}

From points $(2a)$ and $(2b)$  of  \cite[Theorem  10]{DAK_16},  we deduce the following corollary
\begin{cor}
Suppose  there exists $t_0>0$  such  that  $\lambda(t_0)\in \{1,-1\}$.
  If  ${\displaystyle k= \sum_{s=1}^{i-1}l_s+k_i}$  and   $n_i$  is even with  $k_i\leq l_i$,
     then generically  with   respect to   the components of $U$,  the matrix
     $X(t_0)+B(t_0)$   has  the Jordan  canonical form
      \begin{equation}\label{COR}
      \left(\bigoplus_{j=1}^{l_i-k_i} \mathcal{J}_{n_i}(\lambda(t_0))\right)\oplus
  \left(\bigoplus_{j=1}^{l_{i+1}}\mathcal{J}_{n_{i+1}}(\lambda(t_0))\right)
      \oplus\cdots\oplus
      \left(
  \bigoplus_{j=1}^{l_{m(t_0)}}\mathcal{J}_{n_{m(t_0)}}(\lambda(t_0))\right)\oplus \widetilde{\mathcal{J}}(t_0),
      \end{equation}

      where   $\widetilde{\mathcal{J}}(t_0)$   contains all   the Jordan blocks  of   $X(t_0)+B(t_0)$
      associated with eigenvalues  different  from   $\lambda(t_0)$.
\end{cor}

\begin{proof}
\begin{itemize}
  \item If $i=1$, we $k=k_1$  and $n_1$ is even. Thus,  according to $(2a)$ of  Theorem \ref{THR},
  $\widetilde{X}_k(t_0)$ has  the Jordan canonical form

   $$
      \left(\bigoplus_{j=1}^{l_1-k_1} \mathcal{J}_{n_1}(\lambda(t_0))\right)\oplus
  \left(\bigoplus_{j=1}^{l_{2}}\mathcal{J}_{n_{2}}(\lambda(t_0))\right)
      \oplus\cdots\oplus
      \left(
  \bigoplus_{j=1}^{l_{m(t_0)}}\mathcal{J}_{n_{m(t_0)}}(\lambda(t_0))\right)\oplus \widetilde{\mathcal{J}}(t_0),
      $$
      where   $\widetilde{\mathcal{J}}(t_0)$   contains all   the Jordan blocks  of   $X(t_0)+B(t_0)$
      associated with eigenvalues  different  from   $\lambda(t_0)$.
    \item If $i=2$, we have $k=l_1+k_2$ (with $k_2\leq l_2$)  and $n_2$  is even. Thus
    $$
      \widetilde{X}_k(t_0)= \prod_{j=1}^{(k-l_1)}(I+u_{k-j+1}u_{k-j+1}^TJ)
      \underbrace{\prod_{j=1}^{l_1}(I+u_{l_1-j+1}u_{l_1-j+1}^TJ)X(t_0)}_{\widetilde{X}_{l_1}(t_0)}
      $$
      \begin{itemize}
        \item  if $n_1$  is even then according to $(2a)$ of Theorem \ref{THR}, $\widetilde{X}_{l_1}(t_0)$  has
        the Jordan canonical form
          \begin{equation}\label{COP}
  \left(\bigoplus_{j=1}^{l_{2}}\mathcal{J}_{n_{2}}(\lambda(t_0))\right)
      \oplus\cdots\oplus
      \left(
  \bigoplus_{j=1}^{l_{m(t_0)}}\mathcal{J}_{n_{m(t_0)}}(\lambda(t_0))\right)\oplus \widetilde{\mathcal{J}}_{l_1}(t_0),
       \end{equation}
      where $\widetilde{\mathcal{J}}_{l_1}(t_0)$  contains all the Jordan blocks of $\widetilde{X}_{l_1}(t_0)$ associated
      with eigenvalues different from $\lambda(t_0)$. However
        $\widetilde{X}_{k}(t_0)$  is $k_2$ rank one perturbations  of   $\widetilde{\mathcal{J}}_{l_1}(t_0)$ ;
         thus according to $(2a)$ of Theorem \ref{THR}, the Jordan canonical form of $\widetilde{X}_k(t_0)$  is given by
        (\ref{COR}).

        \item if  $n_1$  is odd then according to  $(2b)$ of Theorem \ref{THR}, the  Jordan canonical form of
        $\widetilde{X}_{l_1}(t_0)$ is given by  \ref{COP}. Moreover  $n_2$ being even and $\widetilde{X}_k$ being
        $k$ rank one perturbations of $\widetilde{X}_{l_1}(t_0)$, we obtain that  the Jordan canonical form  of  $\widetilde{X}_k(t_0)$
        is given  by  (\ref{COR})  using $(2a)$  of Theorem \ref{THR}.

      \end{itemize}
  \item If $i>2$ then we have  ${\displaystyle k=\sum_{s=1}^{i-1}l_s+k_i}$ with $k_i\leq l_1$  and $n_i$  even.
  Thus
   $$
      \widetilde{X}_k(t_0)= \prod_{j=1}^{(k-\alpha(i-1))}(I+u_{k-j+1}u_{k-j+1}^TJ)
      \underbrace{\prod_{j=1}^{\alpha(i-1)}(I+u_{\alpha(i-1)-j+1}u_{\alpha(i-1)-j+1}^TJ)X(t_0)}_{\widetilde{X}_{\alpha(i-1)}(t_0)}
      $$
      where ${\displaystyle \alpha(i)=\sum_{j=1}^{i}l_s}$. \\
      From $(2a)$  and $(2b)$  of Theorem \ref{THR} and the above, the Jordan canonical form of
      $\widetilde{X}_{\alpha(i-1)}(t_0)$ is given by  (\ref{FE}).  Thus
       applying   $(2a)$ of Theorem \ref{THR}  to symplectic matrix
       $\widetilde{X}_{\alpha(i-1)}(t_0)$, we obtain  the Jordan canonical form
        (\ref{COR})  of $\widetilde{X}_k(t_0)$.
\end{itemize}

\end{proof}


\section{Algorithm  and numerical examples}\label{Sec5}

We start to recall the following two rotation matrices \cite{kres_05,kres_46}
$$
G_{j,j+N}=\left(
            \begin{array}{ccccc}
              I_{j-1} &  &  &  &  \\
               & \cos(\theta) &  & \sin(\theta) & \\
               &  & I_{N-1} &  &  \\
               & -\sin(\theta) &  & \cos(\theta) &  \\
               &  &  &  & I_{N-j} \\
            \end{array}
          \right), \qquad 1\leq j\leq N,
$$
for some  $\theta \in [-\frac{\pi}{2},\frac{\pi}{2}[ $  and  the direct  sum of two
identical $N\times N$ Householder  matrices
$$
\left(H_j\oplus H_j \right)(v,\,\beta)=\left[
\begin{array}{cc}
I_N-\beta vv^T& \\
&\\
 &I_N-\beta vv^T
\end{array}
\right],
$$
where $v$  is a vector  of length  $N$  with its  first $j-1$ elements
 equal to  zero  and  $\beta$  a scalar   satisfying
  $\beta(\beta v^Tv-2)=0$.  The  symbol $\oplus$  denotes  the  direct
sum  of  matrices.
 From  these matrices, we propose  Algorithm  \ref{Algo}
  which is the synthesis of  Algorithms $23,$ and $24$ of \cite{kres_46}.
  This Algorithm   determines
a basis of an isotropic subspace  from a random matrix.
 \begin{algo}[Computation  of isotropic  subspace]~~\\

 \begin{enumerate}
   \item[] {\bf Imput :} $A\in \Rset^{2N\times k}$, with $N\ge k$.
   \item[] {\bf Output :}   $\mathcal{U}\in \Rset$  isotropic  subspace.
     \begin{enumerate}
       \item $Q =I_{2N}$
       \item for  $j=1,...,k$
       \begin{itemize}
         \item Let $x=Ae_j$
         \item Determine $v\in \Rset^N$  and   $\beta \in \Rset$  such that  the last $N-j$ elements  of
         $$x=\left(H_j \oplus  H_j\right)(v,\beta)x$$  are zero
         \item Determine  $\theta\in [-\frac{\pi}{2},\frac{\pi}{2}[$  such that  $(N+j)$th  element  of
         $x=G_{j,j+k}(\theta)$  is   zero
         \item Dertemine  $\omega\in \Rset^N$  and  $\gamma \in \Rset$  such that  the $(j+1)$th  to  the  kth
         elements  of $$x=\left(H_j \oplus  H_j\right)(\omega,\gamma)x$$  are zero
         \item compute  $ E_j(x)=(H_j\oplus H_j)(v,\beta) G_{j,j+k}
                          (\theta)(H_j\oplus H_j)(w,\gamma)$. \item Put  $A= E^T_j(x)A$  and  $Q=QE_j(x)$
       \end{itemize}
       \item $\mathcal{U}=span(Q(:,1:k))$
     \end{enumerate}
 \end{enumerate}
     \label{Algo}
 \end{algo}

In the following examples  we  show  that any rank k perturbation
 of the solution of (\ref{Eq2}) is  the solution of (\ref{er4}).
 The software used for calculating and plotting the curves of the examples below is
  MATLAB 7.9.0(R2009b).

\begin{exe}\label{e1}
Consider the system of differential equations (see \cite[Vol.2,P.412]{YS})
\begin{equation}\label{Eq11}
\left\{\begin{array}{lll}

q_1\dfrac{d^2\eta_1}{dt^2}+p_1\eta_1+\left[\epsilon\eta_1\cos(\gamma
t)+(\delta\cos(2\gamma
t)+c\sin(2\gamma t))\eta_2\right]=0\\
&&\\
q_2\dfrac{d^2\eta_2}{dt^2}+p_2\eta_2+\left[g\eta_3\sin(5\gamma t)\right]=0\\
&&\\
q_3\dfrac{d^2\eta_3}{dt^2}+p_3\eta_3+\left[(\delta \cos(2\gamma
t)+c\sin(2\gamma t))\eta_1+g\eta_2\sin(5\gamma t)\right]=0

\end{array}
\right.
\end{equation}
 which can be written down as
\begin{equation}\label{Eq12}
\dfrac{d^2\eta}{dt^2}+P(t)\eta=0
\end{equation}
with
$$
\eta=\left(
\begin{array}{c}
\dfrac{\eta_1}{\sqrt{q_1}}\\\\
\dfrac{\eta_2}{\sqrt{q_2}}\\\\
\dfrac{\eta_3}{\sqrt{q_3}}
\end{array}
\right)\quad \text{and}\quad  P(t)=\left(
\begin{array}{ccc}
\frac{p1+\epsilon\cos(\gamma t)}{q_{1}}&0&\frac{\delta\cos(2\gamma
t)+c\sin(2\gamma t)}{\sqrt{q1q3}}\\ \\
0&\frac{p_{2}}{q_{2}}& \frac{g\sin(5\gamma
t)}{\sqrt{q_{1}q_{2}}})\\\\ \frac{\delta\cos(2\gamma
t)+c\sin(2\gamma t)}{\sqrt{q_{1}q_{3}}}&\frac{g\sin(5\gamma
t)}{\sqrt{q_{1}q_{2}}}&\frac{p_{3}}{q_{3}}
\end{array}
\right)
$$

 Putting

$$
X(t)=\left(
\begin{array}{c}
\eta(t)\\
\dfrac{d\eta(t)}{dt}
\end{array}
\right), \quad
 J=\left(
\begin{array}{lll}
0_3& &-I_3\\
I_3&&0_3
\end{array}
\right),\quad \text{and}\quad  H(t)=\left(
\begin{array}{cc}
P(t)&0_3\\
&\\
 0_3&I_3
\end{array}
\right).
$$
 We get a canonical Hamiltonian system
\begin{equation}\label{Eq13}
J\dfrac{dX(t)}{dt}=H(t)X(t),\quad X(0)=I_6
\end{equation}
where $H(t)=H(t+\frac{2\pi}{\sqrt{7}})=H^T(t)$.  In  this example, we take $\gamma=
\sqrt{7},\; q_1=q_2=q_3=1,\;p_1=4,\;p_2=3,\; p_3=2,$
$a=g=\epsilon,\; b=\delta\;\text{and\;} c=0.$

From a random matrix $A\in \Rset^{6\times 3}$, we deduce a matrix $U\in \Rset^{6\times k}$
 of rank $k\leq 3$
whose columns generate an isotropic subspace using algorithm \ref{Algo}

 Consider the perturbed system
(\ref{er4}) of (\ref{Eq13}). We show that the rank $k=2,3$
perturbation of the fundamental solution of (\ref{Eq13}) is the
solution of perturbed system (\ref{er4}). For that, consider

$$
\Psi(t)=\|\widetilde{X}(t)-X_1(t)\|,\, \forall\, t\geq 0
$$
where $\widetilde{X}(t)$ is the solution of (\ref{er4}), and
$X_1(t)=(I+UU^TJ)X(t)$.   We show by numerical examples that
$\Psi(t)$ is very close to zero,
$\forall\,t\in[0,\,\frac{2\pi}{\sqrt{7}}]$.

\begin{itemize}
\item[$\bullet$] for $\epsilon=2$ and $\delta=4,$
consider the random  matrix
$
A={\small \left[
    \begin{array}{ccc}
      0.8147 & 0.2785 & 0.9572 \\
      0.9058 & 0.5469 & 0.4854 \\
      0.1270 & 0.9575 & 0.8003 \\
      0.9134 & 0.9575 & 0.1419 \\
      0.6324 & 0.1576 & 0.4218 \\
      0.0975 & 0.9706 & 0.9157 \\
    \end{array}
  \right].}
$
Using  Algorithm  \ref{Algo} to matrix $A$, we obtain  the matrix
  $
V={\small \left[\begin{array}{ccc}
-0.4918& 0.1282& 0.4009\\
-0.5468& 0.0030& -0.3293\\
-0.0767& -0.6566& 0.1582\\
-0.5514& -0.1635& -0.5002\\
-0.3818& 0.3003 &0.5972\\
-0.0589& -0.6599& 0.3146
 \end{array}\right]}
$
whose columns span  an isotropic subspace.

\begin{itemize}
  \item  Let's take $U=V(:,1:2)$.
In Figure \ref{a0}, we consider the matrix $U$ of rank $2$   which
permits to perturb system (\ref{Eq13}) by the matrices
$U,\; 10^{-1}U,\;10^{-2}U,\;\text{and\;}10^{-3}U.$  We remark that all the
figures  are  so  that  $\Psi(t)\leq 3.5 \times
10^{-14}$.  This proves  that
$\widetilde{X}(t)\equiv X_1(t), \;\forall\;t\in
[0,\;\frac{2\pi}{\sqrt{7}}]$.

\begin{figure}[h!]
\center
\psfig{figure=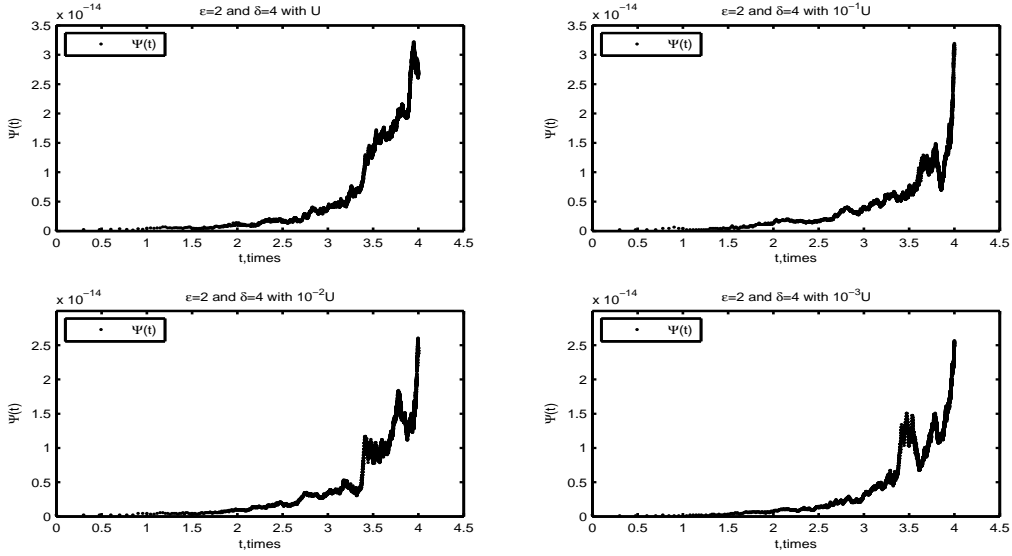,height=80mm,width=160mm}
\caption{Comparison of two solutions (Example 1)}
 \label{a0}
\end{figure}

However, unperturbed system (\ref{Eq13}) is strongly stable. We remark that the rank $2$ perturbed system
 (\ref{er4}) of (\ref{Eq13}) is unstable for the matrix $U$ of rank $2$ and remains strongly stable
 for a matrix  taken in $\{10^{-1}U,\;10^{-2}U,\;10^{-3}U\}$.
  Table \ref{t011} gives the different norms of projectors, the quantity $\delta_S$ and a convergence illustration  of $S^{(n)}.$

 \begin{table}[h!]
\caption{Checking  of the (strong) stability of (\ref{er4}) by the approachs defined in \cite{Dos-Sad_13,dos-cou-sam_13} (Example 1)}
\center
{\small \begin{tabular}{lllllll}
\hline
 & & U  & $10^{-1}U$ & $10^{-2} U$ & $10^{-3}$  & $U\equiv 0$\\
\hline
$\|S^{(n)}\|$ &  &  $5.4202 \times 10^{+33}$  & 7.9357 &  7.9838 & 7.9842 & 7.9842\\
            $\delta_S$ & & -  & 0.3645 & 0.3626 &  0.3625  & 0.3625\\
            \hline
$tr(\PXz)$ &  & $1$ & $-2.6908\times 10^{-34}$ & $5.3580\times 10^{-35}$ & $1.0632\times 10^{-34}$ & $1.9780\times 10^{-34}$\\
$\|\PXz^2-\PXz\|_2$ &  & $1.2741\times 10^{-16}$  & $1.2611\times 10^{-34}$ & $2.7797\times 10^{-34}$ & $1.8430\times 10^{-34}$ & $2.5227 \times 10^{-34}$ \\
$tr(\PXi)$  & & 1 & 0 & 0 & 0 & 0\\
$\|\PXi^2-\PXi\|_2 $ & & $3.8592\times 10^{-16}$ & $2.1197\times 10^{-35}$ & $2.1197\times 10^{-35}$ & $3.3161\times 10^{-35}$ & $2.7261 \times 10^{-35}$\\
$tr(\PXr)$ & & - & 0 & 0 & 0 & 0\\
$\|\PXr^2-\PXr\|_2$ & & - & 0 & 0 & 0 & 0\\
$tr(\PXg)$ & & - & 6 & 6 & 6 & 6\\
$\|\PXg^2-\PXg\|_2$ &  & -  &  $3.4285\times 10^{-19}$ &  $3.4285\times 10^{-19}$ &  $1.1102\times 10^{-16}$ & 0\\
$\|\PXr+\PXg-I_6\|$ & & - & $3.4285\times 10^{-19}$ & $3.4285\times 10^{-19}$ & $1.1102\times 10^{-16}$ & 0\\
            \hline
\end{tabular}}
\label{t011}
\end{table}

Table \ref{t011} justifies the existence of a neighborhood in which any rank $2$ perturbation
of the system remains strongly stable.

\item  In Figure \ref{a1}, we consider $U=V$ to perturb  system (\ref{Eq13}). We can
see  that $\Psi(t)\leq 6\times 10^{-14}$  for all the figures. This shows  that
$\widetilde{X}(t)\equiv X_1(t), \;\forall\;t\in
[0,\;\frac{2\pi}{\sqrt{7}}]$.

\begin{figure}[h!]
\center
\psfig{figure=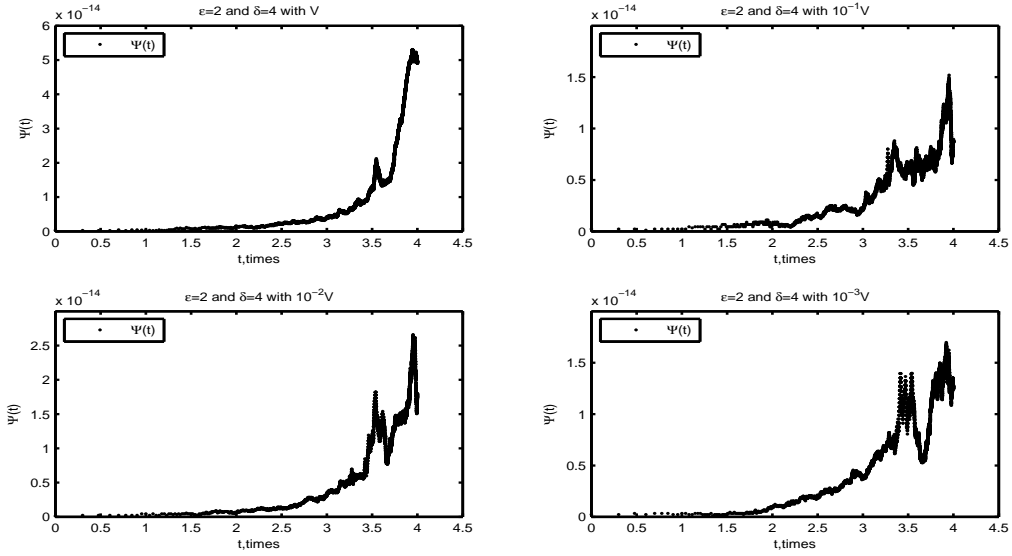,height=80mm,width=160mm}
\caption{Comparison of two solutions (Example 1)}
 \label{a1}
\end{figure}
\end{itemize}

In this example,  the unperturbed system  is strongly  stable for all $U$  taken in  $\{10^{-1}V,10^{-2}V,10^{-3}V\}$
and not stable when $U=V$. This is illustrated in  Table \ref{t012} which gives the norms  of different projectors,
the quantity $\delta$  and a convergence illustration of $S^{(n)}$.

 \begin{table}[h!]
\caption{Checking  of the (strong) stability of (\ref{er4}) by the approachs defined in \cite{Dos-Sad_13,dos-cou-sam_13} (Example 1)}
\center
{\small \begin{tabular}{lllllll}
\hline
 & & U  & $10^{-1}U$ & $10^{-2} U$ & $10^{-3}$  & $U\equiv 0$\\
\hline
$\|S^{(n)}\|$ &  &  $9.1853 \times 10^{+33}$  & 7.9544 &  7.9839 & 7.9842 & 7.9842\\
            $\delta_S$ & & -  & 0.3645 & 0.3626 &  0.3625  & 0.3625\\
            \hline
$tr(\PXz)$ &  & $1$ & $2.4361\times 10^{-35}$ & $1.7323\times 10^{-35}$ & $8.9556\times 10^{-35}$ & $1.9780\times 10^{-34}$\\
$\|\PXz^2-\PXz\|_2$ &  & $1.3362\times 10^{-16}$  & $2.3274\times 10^{-34}$ & $1.4722\times 10^{-34}$ & $1.2520\times 10^{-34}$ & $2.5227 \times 10^{-34}$ \\
$tr(\PXi)$  & & 1 & 0 & 0 & 0 & 0\\
$\|\PXi^2-\PXi\|_2 $ & & $3.6937\times 10^{-16}$ & $2.6155\times 10^{-35}$ & $8.4412\times 10^{-35}$ & $3.4485\times 10^{-35}$ & $2.7261 \times 10^{-35}$\\
$tr(\PXr)$ & & - & 0 & 0 & 0 & 0\\
$\|\PXr^2-\PXr\|_2$ & & - & 0 & 0 & 0 & 0\\
$tr(\PXg)$ & & - & 6 & 6 & 6 & 6\\
$\|\PXg^2-\PXg\|_2$ &  & -  &  $5.4210\times 10^{-20}$ &  $2.5411\times 10^{-21}$ &  $1.1102\times 10^{-16}$ & 0\\
$\|\PXr+\PXg-I_6\|$ & & - & $5.4210\times 10^{-20}$ & $2.5411\times 10^{-21}$ & $1.1102\times 10^{-16}$ & 0\\
            \hline
\end{tabular}}
\label{t012}
\end{table}

The second Table  justifies the existence of a neighborhood in which any rank $3$ perturbation
of the system remains strongly stable.

\item[$\bullet$] for  $\epsilon=15$ and $\delta=4$, we consider the random matrix
$A={\small  \left[
     \begin{array}{ccc}
       0.7482 & 0.8258 & 0.9619 \\
       0.4505 & 0.5383 & 0.0046 \\
       0.0838 & 0.9961 & 0.7749 \\
       0.2290 & 0.0782 & 0.8173 \\
       0.9133 & 0.4427 & 0.8687 \\
       0.1524 & 0.1067 & 0.0844 \\
     \end{array}
   \right]}$.
    Using  Algorithm  \ref{Algo} to  matrix $A$,  we obtain  the matrix
   $V={\small \left[
        \begin{array}{ccc}
          -0.5773 & -0.1332 & 0.4709 \\
          -0.3476 & 0.1520 & 0.1331 \\
          -0.0647 & -0.9504 & -0.2474 \\
          -0.1767 & -0.1538 & 0.6077 \\
          -0.7047 & 0.1706 & -0.5591 \\
          -0.1176 & -0.0103 & -.1320\\
        \end{array}
      \right]}$
      of  rank $3$ whose columns  span  an isotropic subspace.

\begin{itemize}
  \item Let's take $U=V(:,1:2)$.
 Figure \ref{a1} is obtained for values of any matrix of
rank $2$ taken in $\{U,\;10^{-1}U,\;10^{-2}U,\;10^{-3}U\}.$
 We remark that all the figures of Figure
 \ref{a1},   verify
$\Psi(t)\leq 1.4\times 10^{-12}$.
This shows that $\widetilde{X}(t)\equiv X_1(t),\;
\forall\,t\in[0,\;\frac{2\pi}{\sqrt{7}}].$
\begin{figure}[h!]
\center
\psfig{figure=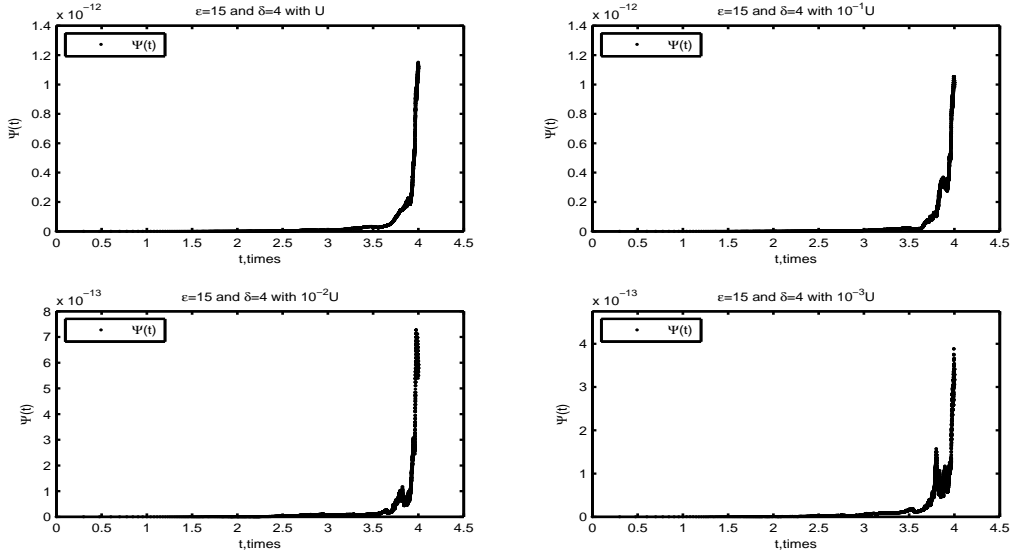,height=80mm,width=160mm}
\caption{Comparison of two solutions.}
 \label{a1}
\end{figure}

In this example, the unperturbed system is unstable, and the rank $2$ perturbation
systems remain unstable for any matrix of rank $2$ taken in $\{U,\;10^{-1}U,\;10^{-2}U,\;10^{-3}U\}$.
This is illustrated in following Table \ref{t013}

\begin{table}[h!]
\caption{Checking  of the (strong) stability of (\ref{er4}) by the dichotomy approach (Example 1)}
\center
{\small \begin{tabular}{lllllll}
\hline
 & & U & $10^{-1}U$ & $10^{-2} U$ & $10^{-3}$  & $U \equiv 0$\\
\hline
$\|S^{(n)}\|$ & & $2.1415\times 10^{+48}$ & $7.8057\times 10^{+32}$ &  $1.8698\times 10^{+35}$ & $1.9709\times 10^{+35}$ & $7.7999\times 10^{+41}$\\
$\delta_S$ & & 0  & 0 & 0 &  0 & 0\\
            \hline
$tr(\PXz)$ & &2 & 2 & 2 & 2 & 2\\
$\|\PXz^2-\PXz\|_2$& & $7.1061 \times 10^{-16}$  & $3.7032 \times 10^{-15}$ & $9.4574 \times 10^{-16}$ & $2.6236 \times 10^{-15}$ & $3.7549 \times 10^{-15}$  \\
$tr(\PXi)$  &  &2 & 2 & 2 & 2 & 2\\
$\|\PXi^2-\PXi\|_2$ & & $1.1448 \times 10^{-15}$ & $7.4439 \times 10^{-15}$ & $4.0245 \times 10^{-15}$ & $2.5933 \times 10^{-15}$ & $3.8816 \times 10^{-15}$\\
\hline
\end{tabular}}
\label{t013}
\end{table}

Thus  there  doesn't  exist  of a neighborhood of the unperturbed system
 in which any rank $2$ perturbation of the system is stable.

 \item
 Taking  $U=V$,
 Figure \ref{a3}  shows  that  $\widetilde{X}(t)\equiv X_1(t),\;
\forall\,t\in[0,\;\frac{2\pi}{\sqrt{7}}]$ for any matrix $U$  of
rank $3$ taken in $\{U,\;10^{-1}U,\;10^{-2}U,\;10^{-3}U\}.$
In the first two  subfigures of Figure \ref{a1}, we see that
$\Psi(t)\leq 1.4\times 10^{-12}$,  while in the other subfigures,
 we note that $\Psi(t)\leq 8 \times 10^{-13}.$

 \begin{figure}[h!]
\center
\psfig{figure=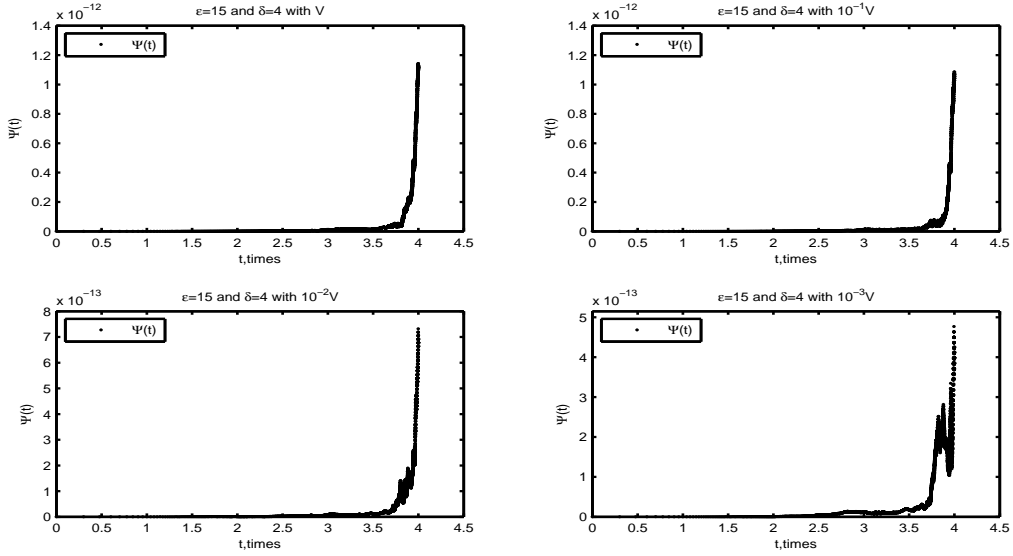,height=80mm,width=160mm}
\caption{Comparison of two solutions.}
 \label{a3}
\end{figure}
 However    Table \ref{t014}   shows  that the perturbed system  is not stable for any matrix
 taken in $\{U,10^{-1}U,10^{-2}U,10^{-2}U,O_6\}$.

 \begin{table}[h!]
\caption{Checking  of the (strong) stability of (\ref{er4}) by the dichotomy approach (Example 1)}
\center
{\small \begin{tabular}{lllllll}
\hline
 & & U & $10^{-1}U$ & $10^{-2} U$ & $10^{-3}$  & $U \equiv 0$\\
\hline
$\|S^{(n)}\|$ & & $1.3786\times 10^{+53}$ & $5.7183\times 10^{+30}$ &  $1.7903\times 10^{+35}$ & $1.9701\times 10^{+35}$ & $1.9720\times 10^{+35}$\\
$\delta_S$ & & 0  & 0 & 0 &  0 & 0\\
            \hline
$tr(\PXz)$ & &2 & 2 & 2 & 2 & 2\\
$\|\PXz^2-\PXz\|_2$& & $4.3586 \times 10^{-16}$  & $9.1634 \times 10^{-16}$ & $9.5301 \times 10^{-16}$ & $2.0073 \times 10^{-156}$ & $2.1088 \times 10^{-15}$  \\
$tr(\PXi)$  &  &2 & 2 & 2 & 2 & 2\\
$\|\PXi^2-\PXi\|_2$ & & $3.8829 \times 10^{-16}$ & $3.8684 \times 10^{-15}$ & $4.5557 \times 10^{-15}$ & $2.0032 \times 10^{-15}$ & $6.0387 \times 10^{-15}$\\
\hline
\end{tabular}}
\label{t014}
\end{table}

\end{itemize}

\end{itemize}

\end{exe}


\begin{exe}\label{e2}
 Consider the following differential system:
\begin{equation}\label{Eqq}
\left\{\begin{array}{lll}
 \dfrac{d^2\beta_1}{dt^2}+\left(4+a\cos(7t)\right)\beta_1+b\beta_3cos(14t)=0\\\\
  \dfrac{d^2\beta_2}{dt^2}+\left(a+b\sin(14t)\right)\beta_2+a\beta_3\sin(35t)=0\\\\
   \dfrac{d^2\beta_3}{dt^2}+3\beta_3+b\beta_1\cos(14t)+a\beta_2\sin(35t)=0
\end{array}
\right.,
\end{equation}
where $a\in\Rset$ and $b\in\Rset^{\star} $ are real parameters. Let
$$
\beta=\left(\begin{array}{c}
\beta_1\\
\beta_2\\
\beta_3

\end{array}
\right)\quad x=\left(\begin{array}{c}
 \beta\\
 \dfrac{d\beta}{dt}
\end{array}
\right).
$$
 System (\ref{Eqq}) can be written as a Hamiltonian  of the form (\ref{Eq2}) with $T=\dfrac{2\pi}{7}$ and

$$
P(t)=\left(\begin{array}{ccc}
4+a\cos(7t)& 0& b\cos(14t)\\
0&a+b\sin(14t)& a\sin(35t)\\
b\cos(14t)& a\sin(35t)& 3
\end{array}
\right)\quad  \text{and} \quad  H(t)=\left(
\begin{array}{ccc}
P(t)& 0_3\\
0_3 & I_3
\end{array}
\right).
$$
We show that the rank $k=2,3$ perturbation of the fundamental solution of (\ref{Eq2})
 is the solution of its rank $k=2,3$ perturbation system. Consider
$$
\Psi(t)=\|\widetilde{X}(t)-X_1(t)\|,\,\forall\, t\in
[0,\;\frac{2\pi}{7}]
$$
where $X_1(t)=(I+UU^TJ)X(t)$ and $\widetilde{X}(t))_{t\in[0,\,\frac{2\pi}{7}]}$ is the solution of the rank $k=2,3$
 perturbation Hamiltonian system (\ref{er4}) of (\ref{Eq2}).
 the following figures represent the norm of the difference between $\widetilde{X}_1(t)$ and $\widetilde{X}(t)$.

 \begin{itemize}
\item[$\bullet$]
for $a=2$ and $b=2,$ consider the  random  matrix
$
A=\left[
    \begin{array}{ccc}
      0.5377 & -0.4336 & 0.7254 \\
      1.8339 & 0.3426 & -0.0631 \\
      -2.2588 & 3.5784 & 0.7147 \\
      0.8622 & 2.7694 & -0.2050 \\
      0.3188 & -1.3499 & -0.1241 \\
      -1.3077 & 3.0349 & 1.4897 \\
    \end{array}
  \right]
$.  Applying Algorithm \ref{Algo} to matrix $A$,  we get the following
matrix
$$ V=\left[
  \begin{array}{ccc}
   -0.1599 & 0.0405 & 0.5357\\
-0.5453 & -0.3844 & 0.2439\\
0.6717 & -0.4143 & 0.1441\\
-0.2564 & -0.7645 & -0.1887\\
-0.0948 & 0.1272 &0.6857\\
0.3889 &-0.2798 &0.3563
  \end{array}
\right]
$$
of rank 3 whose columns generate an isotropic subspace.

\begin{itemize}
  \item
  Considering   the matrix  $U=V(:,1:2)$, we get Figure  \ref{nal1}
  perturbing system  (\ref{Eq2})  by  matrices  taken  in
  $\{U,\;10^{-1}U,\;\;10^{-2}U,\;\;10^{-3}U\}$.

In Figure \ref{nal1}, we note that  all  the figures verify
$\Psi(t)\leq 2.5 \times  10^{-14}$.
This shows that $\widetilde{X}(t)\equiv X_1(t),\;
\forall\;t\in[0,\;\frac{2\pi}{7}].$

\begin{figure}[H]
\center
\psfig{figure=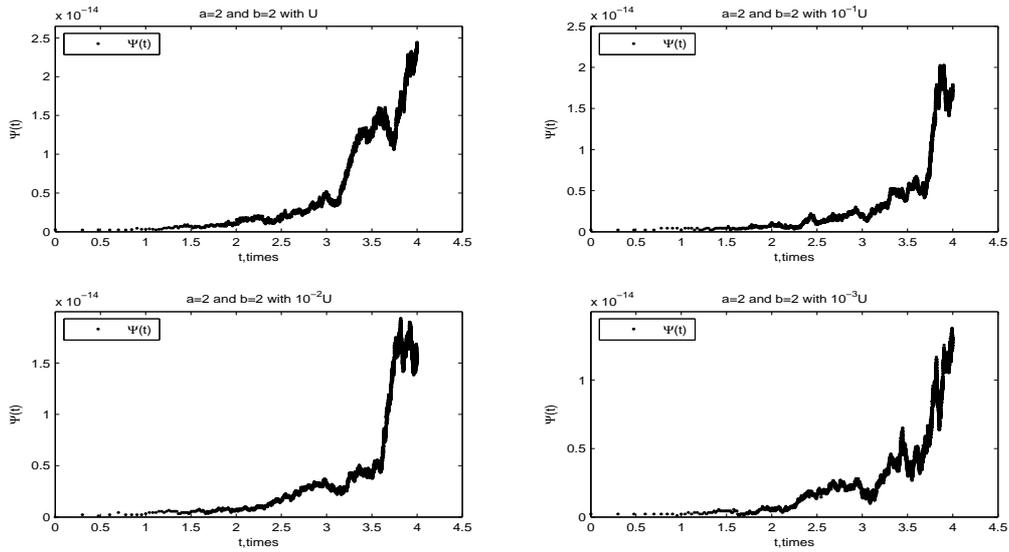,height=80mm,width=160mm}
\caption{Comparison of two solutions (Example 2)}
 \label{nal1}
\end{figure}
In this first example, the unperturbed system is strongly stable and the rank $2$ perturbation
 of the system is also strongly stable for any matrix of rank $2$ belonging to
  $\{10^{-1}U,\;\;10^{-2}U,\;\;10^{-3}U\}$ and is  unstable for any matrix of
  rank $2$  with $U$.  This discussion is summaries in Table $5$

\begin{table}[H]
\label{t015} \caption{Checking  of the (strong) stability of
(\ref{er4}) by the approachs defined in
\cite{Dos-Sad_13,dos-cou-sam_13} (Example 2)} \center {\small
\begin{tabular}{lllllll}
\hline
 & & $U$& $10^{-1}U$ & $10^{-2}U$ & $10^{-3}U$  & $U\equiv 0$\\
\hline
$\|S^{(n)}\|,\;(n=30)  $&& $7.8892$ &$2.1128$&  $2.1115$& $ 2.1115$ & $ 2.1115$\\
$\delta_S$ &&$0.4628$  & $0.2574$  & $0.2528$  & $0.2528$ & $0.2528$ \\
            \hline
$tr(\PXz)$ &  &$-4.2126\times10^{-17}$ &$-1.9678\times10^{-16}$  & $1.3498\times10^{-16}$ & $-2.2171\times10^{-16}$ & $    -4.2126\times10^{-17}$\\
$ \|\PXz^2-\PXz\|_2 $ &  &$1.3434\times10^{-16}  $&$ 1.1411 \times10^{-16} $ & $   1.1709\times10^{-16} $ &$   1.4804\times10^{-16} $&$ 1.3434\times10^{-16} $\\

$tr(\PXi)$ & & $-3.6385\times10^{-17} $ &$-1.0876\times10^{-17} $& $ -9.4959\times10^{-17}$&$-2.2708\times10^{-16}$&$-3.6385\times10^{-17}$\\

$ \|\PXi^2-\PXi\|_2 $ & & $1.4218\times10^{-16} $&$ 1.4991\times 10^{-16} $ & $   1.0851\times10^{-16}$ &$ 1.7867\times10^{-16}$ &$ 1.4218\times10^{-16} $\\

$tr(\PXr)$&  & $6$  & $6 $&$6 $&$6$ &$ 6$\\

$ \|\PXr^2-\PXr\|_2$&  &$0$&$0$ &$0$&$0$ & $0$ \\

$tr(\PXg) $&  &$0$  &$ 0 $  &$ 0 $  &$0$   & $0$\\
$ \|\PXg^2-\PXg\|_2$&  &$0$  &$ 0 $  &$ 0 $  &$0$ & $0$\\

$\|\PXr+\PXg-I_6\|_2$ &  &$0$  &$0 $  &$ 0 $  &$0$& $0$\\
            \hline
\end{tabular}}
\end{table}
This justifies the existence of a neighborhood of the unperturbed
system in which any  random rank $2$ perturbation of the system
remains strongly stable.
\item Let's take $U=V$ ;  Figure \ref{ami1} shows that
$\Psi(t)<3.5\times10^{-14},\,\forall\, t\in [0,\;\frac{2\pi}{7}]$
for all  the  figures. This shows that $\widetilde{X}(t)\equiv X_1(t).$
\begin{figure}[H]
\center \psfig{figure=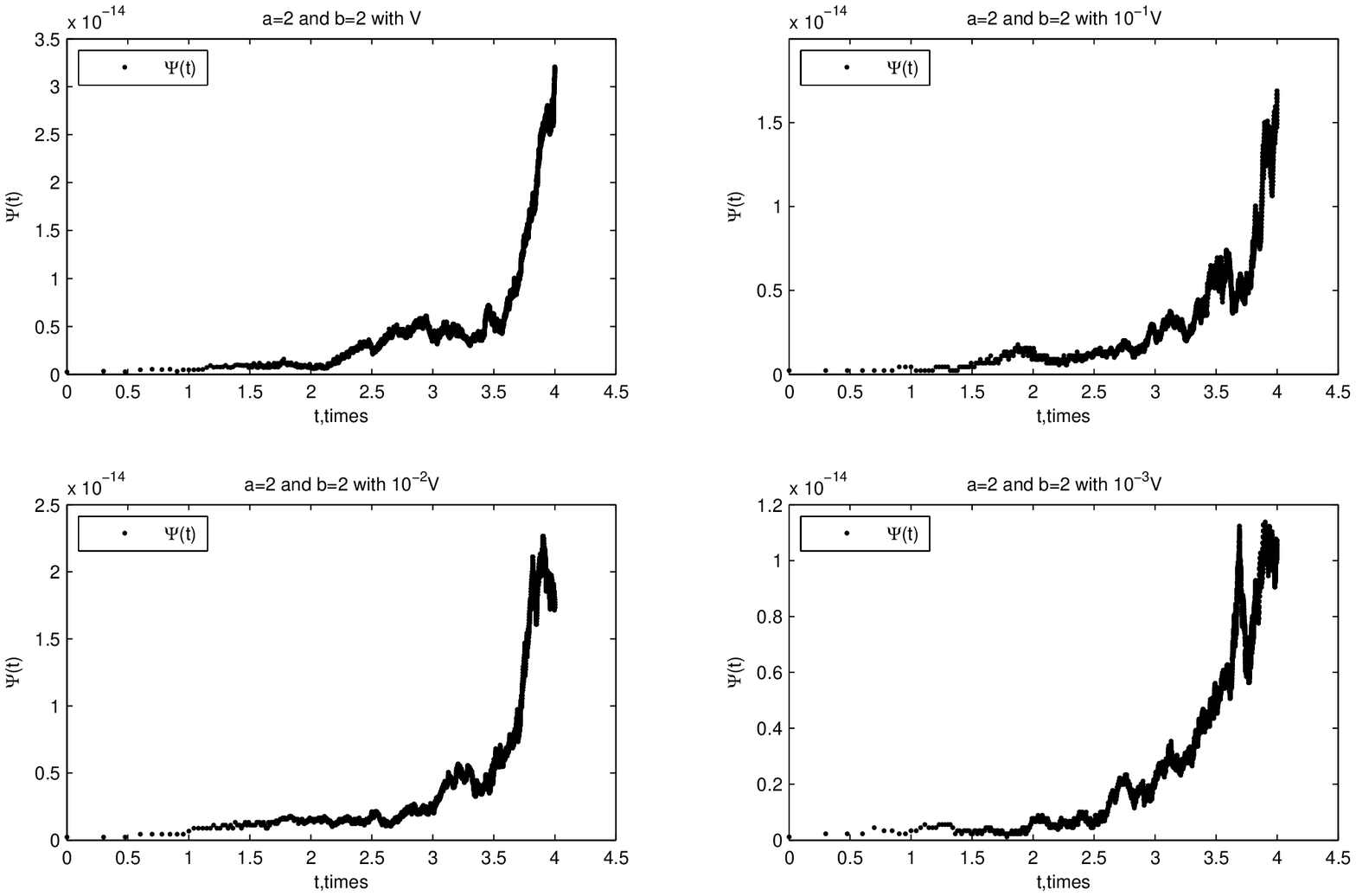,height=80mm,width=160mm}
\caption{Comparison of two solutions (Example 2)}
 \label{ami1}
\end{figure}
In this case, the unperturbed system (\ref{Eq2}) is strongly stable
for  any random matrix $U$ of rank $3$ belonging to
$\{V,\;10^{-1}V,\;10^{-2}V,\,10^{-3}V\}$.  This is illustrated in  Table \ref{t016}
\begin{table}[H]
\caption{Checking  of the (strong) stability of (\ref{er4}) by the
approachs defined in \cite{Dos-Sad_13,dos-cou-sam_13}  (Example 2)}
 \center {\small
\begin{tabular}{lllllll} \hline
 & & $V$& $10^{-1}V$ & $10^{-2} V$ & $10^{-3}V$  & $V\equiv 0$\\
\hline
$\|S^{(n)}\|,\;(n=30)  $&& $11.8852$ &$2.1209$&  $2.1116$& $ 2.1115$ & $ 2.1115$\\
$\delta_S$ &&$0.3599$  & $0.2532$  & $0.2528$  & $0.2528$ & $0.2528$ \\
            \hline

$tr(\PXz)$ &  &$ -2.0382\times10^{-16}$ &$ 1.2364\times10^{-16}$  & $ 1.0578\times10^{-18}$ & $ -9.0143\times10^{-17}$ & $    -4.2126\times10^{-17}$\\
$ \|\PXz^2-\PXz\|_2 $ &  &$ 1.1186\times10^{-16}  $&$  1.7572\times10^{-16} $ & $   5.8457\times10^{-17} $ &$    8.4476\times10^{-17} $&$ 1.3434\times10^{-16} $\\

$tr(\PXi)$ & & $ -2.2819\times10^{-16} $ &$ -1.0701\times10^{-17} $& $  -6.8127\times10^{-19}$&$    -1.0956\times10^{-17}$&$-3.6385\times10^{-17}$\\

$ \|\PXi^2-\PXi\|_2 $ & & $  1.1307\times10^{-16} $&$ 2.5771\times 10^{-16} $ & $    6.3190\times10^{-17}$ &$   1.1784 \times10^{-16}$ &$ 1.4218\times10^{-16} $\\

$tr(\PXr)$&  & $6$  & $6 $&$6 $&$6$ &$ 6$\\

$ \|\PXr^2-\PXr\|_2$&  &$0$&$0$ &$0$&$0$ & $0$ \\

$tr(\PXg) $&  &$0$  &$ 0 $  &$ 0 $  &$0$   & $0$\\
$ \|\PXg^2-\PXg\|_2$&  &$0$  &$ 0 $  &$ 0 $  &$0$ & $0$\\

$\|\PXr+\PXg-I_6\|_2$ &  &$0$  &$0 $  &$ 0 $  &$0$& $0$\\

            \hline

\end{tabular}}
\label{t016}
\end{table}
This justifies the existence of a neighborhood of the unperturbed
system in which any  random rank $3$ perturbation of the system
remains strongly stable.
\end{itemize}

\item[$\bullet$] For $a=18.95$ and $b=2,$ consider the random matrix
$$
A= \left[
\begin{array}{ccc}
1.4090&0.4889&0.8884\\
1.4172&1.0347&-1.1471\\
0.6715&0.7269&-1.0689\\
-1.2075&-0.3034&-0.8095\\
0.7172&0.2939&-2.9443\\
1.6302&-0.7873&1.4384
\end{array}
\right].
$$
Applying the algorithm \ref{Algo} to the matrix $A$, we have the following
random matrix
$$
V=\left[
\begin{array}{ccc}
 -0.4677&-0.3232&0.6729\\
 -0.4704&-0.4311&-0.3311\\
 -0.2229&-0.1819&0.2867\\
 0.4008&-0.1814&0.1744\\
 -0.2381&-0.3203&-0.5692\\
 -0.5412&0.7356&-0.0323
\end{array}
\right]
$$
of rank $3$ whose columns generate an isotropic subspace.
 \begin{itemize}
\item Let's take $U=V(:,1:2)$.   The following Figure shows that
$\widetilde{X}(t)\equiv
X_1(t),\;\,\forall\;t\in[0,\;\frac{2\pi}{7}]$,  for any matrix of
rank $2$ belonging to $\{U,\;10^{-1}U,\;10^{-2}U,\;10^{-3}U\}.$
Thus  in Figure \ref{ami4}, we can observe that $\Psi(t)\leq
1\times10^{-13},\,\forall\, t\in [0,\;\frac{2\pi}{7}]$ for all
figures.
\begin{figure}[H]
\center \psfig{figure=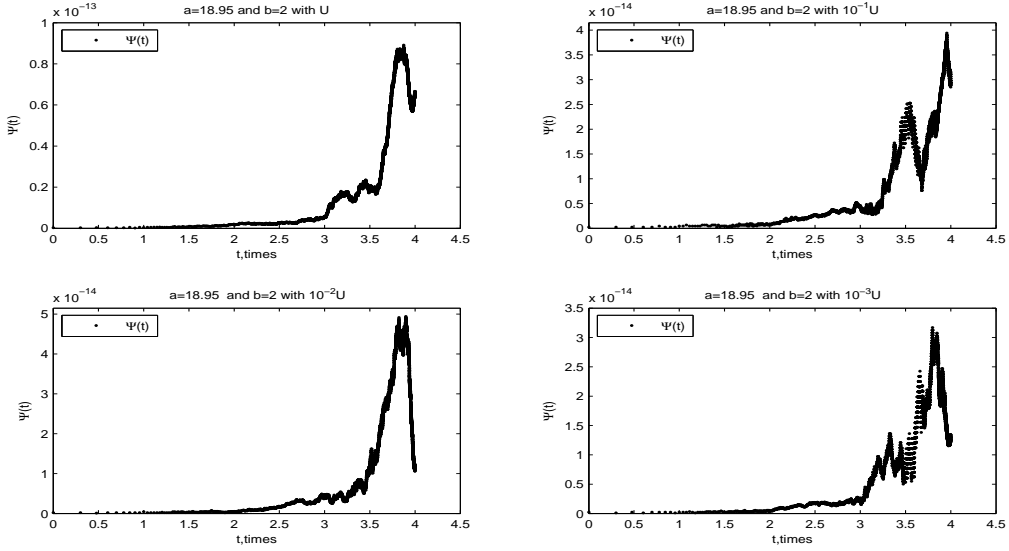,height=80mm,width=160mm}
\caption{Comparison of two solutions (Example 2)}
 \label{ami4}
\end{figure}
In this case, the unperturbed system is unstable and its rank $2$
perturbation systems remain unstable for any matrix of rank $2$
taken in $\{U,\,10^{-1}U,\;10^{-2}U,\;10^{-3}U\}$.  This is
illustrated in Table $\ref{t017}$
\begin{table}[H]

\caption{Checking  of the (strong) stability of (\ref{er4}) by the
approachs defined in \cite{Dos-Sad_13,dos-cou-sam_13} (Example 2)}
\label{t017} \center {\small
\begin{tabular}{lllllll} \hline
 & & $U$ & $10^{-1}U$ & $10^{-2} U$ & $10^{-3}U$  & $U\equiv 0$\\
\hline
$\|S^{(n)}\|,\;(n=30)  $&& $ 4.9239\times 10^{+57}$ &$5.7014\times 10^{+45}$&  $ 5.2235\times 10^{+45}$& $  5.2189\times 10^{+45}$ & $5.2189\times 10^{+45}$\\
            \hline
$tr(\PXz)$ &  &$2$ &$1$  & $1$ & $1$ & $1$\\
$ \|\PXz^2-\PXz\|_2 $ &  &$ 7.3293\times10^{-16} $&$ 3.0130\times10^{-16} $ & $ 1.8699\times10^{-16} $ &$  1.8367\times10^{-16} $&$7.4397\times10^{-17} $\\

$tr(\PXi)$ & & $2 $ &$1 $& $ 1$&$1 $&$1$\\

$ \|\PXi^2-\PXi\|_2 $ & & $1.6104\times10^{-15} $&$2.0658\times 10^{-16} $ & $ 2.2422\times10^{-16}$ &$1.0577\times10^{-15}$ &$1.1106\times10^{-15} $\\
\hline

\end{tabular}}
\label{t017}
\end{table}
This justifies the existence of a neighborhood of the unperturbed
system in which any rank $2$ perturbation of the system remains
unstable.
\item In this latter example, we consider $U=V$ to perturb system (\ref{Eq2}).
 Figure \ref{ami3} is obtained for value of any random matrix $U$ of rank $3$ taken
in $\{U,\,10^{-1}U,\,10^{-2}U,\,10^{-3}U\}$.  We can see that
$\Psi(t)\leq 5.25\times10^{-14},\,\forall\, t\in
[0,\;\frac{2\pi}{7}]$ for all figures. Hence, we have
$\widetilde{X}(t)\equiv
X_1(t),\;\,\forall\;t\in[0,\;\frac{2\pi}{7}]$.
\begin{figure}[H]
\center \psfig{figure=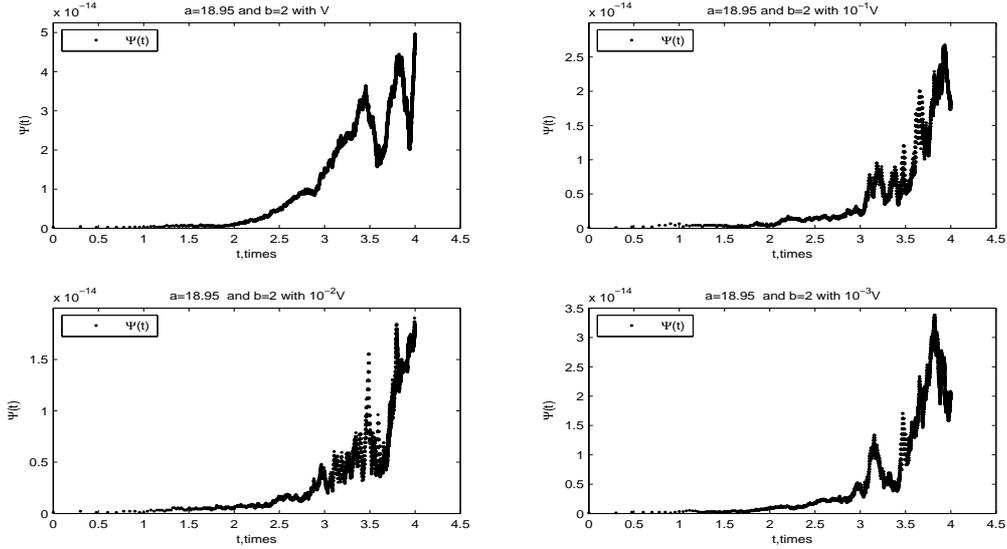,height=80mm,width=160mm}
\caption{Comparison of two solutions (Example 2)}
 \label{ami3}
\end{figure}
However the following Table $\ref{t018}$ shows that the perturbed system is
not stable for any random matrix $U$ of rank $3$ taken in
$\{U,\;10^{-1}U,\;10^{-2}U,\;10^{-3}U,\,O_{6,3}\}$.
\begin{table}[H]
\caption{Checking  of the (strong) stability of (\ref{er4}) by the
approachs defined in \cite{Dos-Sad_13,dos-cou-sam_13} (Example 2)}
 \center {\small
\begin{tabular}{lllllll}

\hline
 & & $U$  & $10^{-1}U$ & $10^{-2} U$ & $10^{-3}U$  & $U\equiv 0$\\
\hline
 $\|S^{(n)}\|,\;(n=30)  $&& $1.6182\times 10^{+55}$ &$7.3921\times 10^{+45}$&  $ 5.2372\times 10^{+45}$& $ 5.2191\times 10^{+45}$ & $5.2189\times 10^{+45}$\\
            \hline
$tr(\PXz)$ &  &$2$ &$1$  & $1$ & $1$ & $1$\\
$ \|\PXz^2-\PXz\|_2 $ &  &$ 2.0170\times10^{-16} $&$ 1.3412\times10^{-16} $ & $ 2.9536\times10^{-16} $ &$  2.0371\times10^{-16} $&$7.4397\times10^{-17} $\\

$tr(\PXi)$ & & $2 $ &$1 $& $ 1$&$1 $&$1$\\

$ \|\PXi^2-\PXi\|_2 $ & & $8.5346\times10^{-16} $&$5.0453\times 10^{-16} $ & $4.8715\times10^{-16}$ &$1.0011\times10^{-15}$ &$1.1106\times10^{-15} $\\
            \hline
\end{tabular}}
\label{t018}
\end{table}
\end{itemize}

\end{itemize}
\end{exe}
%
%

\section{Concluding remarks}\label{Sec6}

In this research work, after defining a rank $k$ perturbation theory
  of a Hamiltonian system with periodic coefficients  with $k\ge 2$,
we showed that the solution of its rank $k$ perturbation  is
the same as the rank $k$ perturbation of the solution of unperturbed  system.
   Then we analyzed Jordan canonical form of the solution of
    the unperturbed system when it is subjected to a rank $k$  perturbation.
    This analysis is  a generalization of that made by  M.  Dosso, et al. in  \cite{DAK_16} in the case of a
     rank  one pertubation of Hamiltonian system  with
    periodic  coefficients.
    Finally we proposed numerical examples which confirm this theory.
    However, these examples use an algorithm that randomly
    constructs an isotropic subspace basis.
     From these numerical examples we notice that when a
      system is strongly stable (respectively unstable),
      there exists a neighborhood in which any rank $k$ perturbation
       of the system in this neighborhood remains
       strongly stable (respectively unstable)

      In future work, we will compare the zone of stability (strong)
       of the Hamiltonian systems with periodic coefficients
        and their  rank $k\ge 1$ perturbations.
        Then it would be boring to find a link between any random perturbation and  rank $k\ge 1$ perturbation  of
        Hamiltonian system with periodic coefficients.


\begin{thebibliography}{99}
 \bibitem{Brez_02}
C. Brezinski, Computational Aspects of Linear Control, Kluwer Academic
Publishers, 2002.
\bibitem{DAK_16}
M.Dosso, T.G.Y. Arouna., and  J.C.Koua. Brou, \emph{On rank one
perturbations of Hamiltonian system with periodic coefficients.}
Wseas Translations on Mathematics. Volume 15, 2016, Pages 502-510.
\bibitem{Dos-Sad_13}
M. Dosso  and M. Sadkane. On the strong stability of
symplectic matrices.
 Numerical Linear Algebra with Applications 20(2) (2013), 234-249.


\bibitem{dos-cou_14}   
M.Dosso and N. Coulibaly, \emph{Symplectic matrices and strong stability of
Hamiltonian systems with periodic coefficients.} Journal of
Mathematical Sciences : Advances and Applications. Vol. 28(2014), pp 15-38.
\bibitem{dos-cou-sam_13}
M. Dosso, N. Coulibaly and L. Samassi, Strong stability of symplectic matrices using a spectral dichotomy method,
Far East Journal  Applied Mathematics. Vol. 79, No 2, 2013, pp 73-110.
\bibitem{fre-meh-xu_02}   
G. Freiling, V. Mehrmann, and H. Xu. Existence, \emph{uniqueness and
parametrization of Lagrangian invariant subspaces.} SIAM J. Matrix
Anal. Appl., 23:1045{1069, 2002.}
\bibitem{god_92}
Godunov  SK, Verification  of boundedness for  the powers  of symplectic  matrices
with   the help  of  averaging.  Siber. Math. J. 1992 ; 33 : 939-949.
\bibitem{god_89} 
S.K. Godunov, Stability of iterations of symplectic transformations,
Siberian Math. J. 30, 54-63 (1989).

\bibitem{god-sad_01}
S.K. Godunov, M. Sadkane,  Numerical determination of a canonical
form of a symplectic matrix, Siberian Math. J.  42, 629-647 (2001)
\bibitem{god-sad_06}
S.K. Godunov, M. Sadkane,  Spectral analysis of symplectic matrices
with application to the theory of parametric resonance, SIAM J.
Matrix Anal. Appl. 28, 1083-1096  (2006)
\bibitem{goh-lan-rod_05}
I. Gohberg, P. Lancaster, and L. Rodman. Indefinite Linear Algebra and
Applications. Birkhäuser, Basel, 2005.

\bibitem{hass_99}
B. Hassibi, A.  H.  Sayed, T.  Kailath, \emph{Indefinite-Quadratic
Estimation   and Control,} SIAM, Philadelphia, PA, 1999.

\bibitem{kres_05}
D. Kressner, Perturbation  bounds for isotropic invariant subspaces of skew-Hamiltonian matrices.
 SIAM J.  Matrix Analysis  Applications  26(4): 947-961, 2005.

\bibitem{kres_46}
D. Kressner, Numerical  Methods for  General  and Structured Eigenvalue  Problems.
Lecture  Notes  in Computational Science  and Engineering  46, Springer 2005.
  ISBN 978-3-540-24546-9, pp. I-XIV, 1-264.

\bibitem{lan-rod_95}
P. Lancaster and L. Rodman. The Algebraic Riccati Equation. Oxford University
Press, Oxford, 1995.


\bibitem{Meh}
C. Mehl, V. Mehrmann,  A.  C. M. Ran  and L. Rodman.  Eigenvalues
Perturbation  theory  of  structured  matrices  under  generic  structured  rank
one perturbations ;  Symplectic, othogonal  and unitary  matrices. BIT, 54(2014), 219-255.
\bibitem{Meh1}
C. Mehl, V. Mehrmann,  A.  C. M. Ran  and L. Rodman.
Perturbation analysis of Lagrangian invariant subspaces of symplectic matrices.
Linear and Multilinear Algebra, 57:141-184, 2009.
\bibitem{Saad_11}
Y. Saad, Numerical methods for large eugenvalue  problems.
 SIAM, 2nd ed., 2011.
\bibitem{Wat07}
D. S. Watkins, The matrix eigenvalue problem.
 GR and Krylov Subspace Methods
,SIAM,  Philadelphia, 2007.
\bibitem{YS}
V.A. Yakubovich,  V.M. Starzhinskii,
\newblock{ Linear differential equations with periodic coefficients},
\newblock Vol. 1 \& 2., Wiley, New York  (1975)
\bibitem{Yan2}
YAN  Qing-you.  The properties  of  a kind  og random  symplectic  matrices.
 Applied  mathematics  and Mechanics. Vol  23, No 5, May  2002.
\end{thebibliography}
\end{document}